\RequirePackage{fix-cm}
\documentclass[smallcondensed]{svjour3}       
\smartqed  
\usepackage{graphicx}
\usepackage{amsmath}
\usepackage{amssymb}

\usepackage[margin=3.25cm]{geometry}

\usepackage{subcaption}
\usepackage[shortlabels]{enumitem}
\usepackage[citecolor=blue, linkcolor=red, colorlinks=true]{hyperref}

\newcommand{\im}[1]{\mathbf{ i_{#1}}}
\newcommand{\ims}[1]{\mathbf{ i_{#1}^2}}
\newcommand{\jm}[1]{\mathbf{ j_{#1}}}
\newcommand{\jms}[1]{\mathbf{ j_{#1}^2}}

\newcommand{\gm}[1]{\gamma_{#1}}
\newcommand{\gmc}[1]{\overline{\gamma}_{#1}}

\newcommand{\be}[1]{\mathbf{e_{#1}}}
\newcommand{\bec}[1]{\mathbf{\overline{e}_{#1}}}

\newcommand{\um}[1]{\mathbf{u_{#1}}}
\newcommand{\ums}[1]{\mathbf{u_{#1}^2}}
\newcommand{\vm}[1]{\mathbf{v_{#1}}}
\newcommand{\vms}[1]{\mathbf{v_{#1}^2}}

\newcommand{\mC}{\mathbb{C}}

\newcommand{\mN}{\mathbb{N}}

\newcommand{\mR}{\mathbb{R}}
\newcommand{\mT}{\mathbb{T}}

\newcommand{\mM}{\mathbb{M}}
\newcommand{\mI}{\mathbb{I}}
\newcommand{\mS}{\mathbb{S}}

\newcommand{\mL}{\mathbb{L}}

\newcommand{\cK}{\mathcal{K}}

\newcommand{\cM}{\mathcal{M}}

\newcommand{\cT}{\mathcal{T}}

\newcommand{\ra}{\rightarrow}

\DeclareMathOperator{\spn}{span}

\begin{document}

\title{Classification of Principle 3D Slices of Filled-in Julia Sets in
Multicomplex Spaces
}


\author{Quentin Charles         \and
        Pierre Olivier Parisé 
}


\institute{Quentin Charles \at
              University of Hawai'i at Manoa\\
              2500 Campus Rd, Honolulu, HI, 96822 \\
              \email{qcharles@hawaii.edu}           
           \and
           Pierre-Olivier Parisé \at
              Université du Québec à Trois-Rivières\\ 
              3351 Bd des Forges, Trois-Rivières, QC, G8Z 4M3, Canada\\
              \email{pierre-olivier.parise@uqtr.ca}
}

\date{Received: date / Accepted: date}

\maketitle

\begin{abstract}
A generalization of the filled-in Julia set is presented using the multicomplex numbers and an algorithm is presented to visualize these sets in the tridimensional space. There are many ways to visualize these higher dimensional fractals sets on a computer. We therefore introduce an equivalence relation between 3D representations and show that, for the filled-in Julia sets associated to the polynomial $z^p + c$, there are nine 3D slices when $p$ is an odd integer and four when $p$ is even. These results differs from the recent characterization obtained by Brouillette and Rochon in 2019 and the proofs require different arguments in the context of the filled-in Julia sets.
\end{abstract}

\section{Introduction}

The theory of dynamical systems is a branch of mathematics that studies the evolution of a system that follows specific rules. This discipline was initiated by H. Poincaré following his study of the three-body problem. It became popular among the public thanks to the experiments by E. Lorenz on climate models in 1963, which also gave rise to the expression ``butterfly effect.'' Today, this mathematical theory has become a central tool in many fields of Science and engineering \cite{Mandel1975,May1976}.

One branch of the theory of dynamical systems is holomorphic dynamics, which studies the dynamics generated by the iteration of differentiable functions with complex values. It arose from the work of two French mathematicians, Gaston Julia and Pierre Fatou, published between 1917 and 1920. The work of mathematicians such as B. Mandelbrot, A. Douady, and J. H. Hubbard revived researchers' interest in this theory. Among other things, two sets in particular have attracted attention due to their remarkable properties: the filled-in Julia sets and the Mandelbrot set.

On one hand, given a polynomial function $f (z) = a_n z^n + a_{n-1} z^{n-1} + \cdots + a_1 z + a_0$, the filled-in Julia set associated to the function $f$ is the set of $z$ such that the orbit of $z$ under $f$, that is the sequence $(f^m (z))_{m = 1}^\infty$, where $f^m (z) := f (f^{m-1} (z))$, is bounded. In this paper, we focus mainly on the filled-in Julia sets associated to the polynomials $f_c (z) = z^p + c$, where $p \geq 2$ is an integer. These sets will be denoted by $\cK_{1, c}^p$. 
On the other hand, the Mandelbrot set is the set of parameters $c$ such that the orbit of $0$ is bounded, that is the sequence $(f_c^m (0))_{m = 1}^\infty$ remains bounded. It is usually denoted by $\cM_1^p$.

In recent years, more and more researchers became interested in generalizing the Mandelbrot set and the filled-in Julia sets to higher dimensions. One of the pioneers in this field was A. Norton \cite{Norton1982}. He used the quaternions to render a three dimensional version of a filled-in Julia set. Many others followed Norton's work, but the generalization that gave rise to the richer theory was obtained by Rochon \cite{rochon2000,rochon2003}, Garrant-Pelletier and Rochon \cite{PR2009}, and Parisé and Rochon \cite{PR2015}. They used the so-called \textit{multicomplex numbers}. 

Multicomplex numbers, usually denoted by $\mM \mC (n)$ or $\mM \mC_n$ where $n \geq 0$ is an integer representing the order, are a commutative generalization of complex numbers to higher dimensions that were introduced by C. Segre. They are constructed inductively with $\mM (0) := \mR$ and $\mM \mC (1)$ being an isomorphic copy of the complex numbers.  The main reference for multicomplex numbers is Price's book \cite{Pr1991}. We have seen, in recent years, an explosion of articles referring to applications of the multicomplex numbers, including the generation of 3D fractals. Here is a non exhaustive list of applications of the multicomplex numbers: generalization of the Mandelbrot and filled-in Julia sets to higher dimensions \cite{Brouillette_2019,PR2009,PR2015,rochon2000,rochon2003,Wang2013}, generalization of some equations in physics such as the linear and non-linear Schr{\"o}dinger equations and solitons \cite{B2017,Cen2020,K2023,rochon2004,TVG2017}, and generalization of complex-valued neural networks \cite{Alpay2023}. For instance, in \cite{rochon2004}, the author employed bicomplex numbers to generalize the Schrödinger equation. As a result, he was able to unify two fundamental equations in physics---the Schrödinger equation and the Klein-Gordon equation---within a single framework.

When $n = 2$, the set $\mM \mC (2)$ is the set of bicomplex numbers and these numbers take the form 
    $$  
        a + b \im{1} + c \im{2} + d \jm{1}
    $$
where $a, b, c, d \in \mR$, $\ims{1} = \ims{2} = -1$ with $\im{1} \neq \im{2}$ and $\jms{1} = 1$ with $\jm{1} \neq 1$. Tricomplex numbers have a similar representation but with $8$ components and in general, a multicomplex numbers have a similar representation but with $2^n$ components. Therefore a Mandelbrot set or a filled-in Julia set generalized using multicomplex numbers can't be visualize when the order of the multicomplex numbers used is greater than or equal to $2$. Garrant-Pelletier and Rochon \cite{PR2009} solved this problem by considering three dimensional subspaces of the multicomplex numbers to visualize the  Mandelbrot set and the theory was then further generalized by Parisé and Rochon \cite{PR2015}, and improved by Brouillette and Rochon \cite{Brouillette_2019}.

There is a remarkable phenomena that occurs in the case of the multicomplex Mandelbrot set. Many visualization through different three dimensional subspaces, called principal 3D slices, give rise to the same set and therefore the multicomplex Mandelbrot set exhibits internal symmetries at lower dimensions. In particular, Brouillette and Rochon have shown the following important result.

\begin{theorem}
    Let $p \geq 2$ and $n \geq 3$. Then any principal 3D slices of the multicomplex Mandelbrot set can be obtained from a principal 3D slices of the tricomplex Mandelbrot set. 
\end{theorem}
No such result exists for the multicomplex filled-in Julia sets, except some rudimentary results on the bicomplex filled-in Julia sets obtained in \cite{Wang2013}.
The main goal of this paper is to therefore introduce the multicomplex filled-in Julia sets associated to the polynomials $f_c$ and to prove a quantitive result on the principal 3D slices which is the quantitative analogue of Brouillette and Rochon's result. 
Our main result reads as followed:
\begin{theorem}\label{Thm:main}
    Let $p \geq 2$,  $n \geq 3$ be integers, and $c$ be a real number.
        \begin{enumerate}
            \item If $p$ is an even integer, then the multicomplex filled-in Julia set has $4$ principal 3D slices for any $n$.
            \item If $p$ is an odd integer and $c = 0$, then the multicomplex filled-in Julia set has $4$ principal 3D slices.
            \item If $p$ is an odd integer and $c \neq 0$, then the multicomplex Mandelbrot set has $8$ principal 3D slices when $n = 3$ and $9$ principal 3D slices for any $n > 3$. 
        \end{enumerate}
\end{theorem}

The paper is structured as followed. In Section \ref{Sec:PrelimMulticomplex}, we introduce the algebra of multicomplex numbers. In Section \ref{Sec:ComplexFilledJ}, we introduce the precise definition of the filled-in Julia set and give some of their basic properties. In Section \ref{Sec:MulticomplexFilledJ}, we introduce the multicomplex filled-in Julia sets and present some of their elementary properties. In Section \ref{Sec:Visualisation3D}, we define what is a principal 3D slices of a filled-in Julia set. In section \ref{Sec:DefinitionEquivalence3D}, we introduce our new equivalence relation between the sets of principal 3D slices of a filled-in Julia set and prove a list of preliminary results needed in the proof of our main result. Finally, in Section \ref{Sec:Characterization3DSlices}, we prove our main result on the number of principal 3D slices of the multicomplex filled-in Julia set.

\section{Preliminary on Multicomplex Numbers}\label{Sec:PrelimMulticomplex}
The objective of this section is to introduce the definition of multicomplex numbers, their basic algebraic operations, and their topology. 

\subsection{Multicomplex Numbers}
The Multicomplex number are a generalization of the complex numbers to higher dimensions. The original reference to multicomplex numbers is the book of Price \cite{Pr1991}. For recent and succinct introductions to multicomplex numbers, the reader is also directed to the papers \cite{Brouillette_2019,doyon2022}. Our presentation will mainly follow the paper \cite{doyon2022}.

The set of Multicomplex numbers is denoted by $\mathbb{M} \mC (n)$ for any $n \geq 1$ , where $n$ is an integer representing the order. For $n = 0$, we set $\mM \mC (0) := \mR$ , for $n = 1$, we set $\mM \mC (1) := \{ \eta_1 + \eta_2 \im{1} \, : \, \eta_1, \eta_2 \in \mR \}$ with $\im{1} = \im{}$, the usual imaginary unit, and for any integer $n \geq 2$, the element of $\mM \mC (n)$ are defined recursively as followed:
    $$
        \mathbb{M} \mC (n):=\{\eta_1+\eta_2\im{n} \, : \, \eta_1,\eta_2 \in \mathbb{M} \mC (n-1)\} ,
    $$
where the symbol $\im{n}$ is an imaginary unit such that $\ims{n} = -1$ and $\im{n-1} \neq \im{n}$. Note that with this definition, the usual set of complex numbers $\mathbb{C}$ corresponds to the set $\mathbb{M} \mC (1)$.

Given two multicomplex numbers $\eta = \eta_1 + \eta_2 \im{n}$ and $\zeta = \zeta_1 + \zeta_2 \im{n}$, the operation of addition, denoted $+$, of two multicomplex numbers is defined as followed:
    $$
        \eta + \zeta = ( \eta_1 + \zeta_1 ) + ( \eta_2 + \zeta_2 ) \im{n} ,
    $$
where the addition in each component is the addition defined in $\mM  \mC (n - 1)$ and the operation of multiplication, denoted by $'\cdot'$, between two multicomplex numbers is defined as
    $$
        \eta \cdot \zeta := ( \eta_1 \zeta_1 - \eta_2 \zeta_2 ) + ( \eta_1 \zeta_2 + \eta_2 \zeta_1 ) \im{n},
    $$ 
where the operations in each components are the addition and multiplication coming from $\mM \mC (n -1)$. One can show that the triplet $(\mathbb{M} \mC (n), \cdot, +)$ forms a commutative unitary ring (see, for instance, \cite{Pr1991}). 

When $n = 1$, we obtain the set of complex numbers $\mM \mC (1)$ which, compared to all other cases of multicomplex numbers, is a commutative field. When $n = 2$, we obtain the set of bicomplex numbers and each bicomplex number can be written down as $\eta = \eta_1 + \eta_2 \im{2}$, where $\eta_1 , \eta_2 \in \mM (1)$. Describing each component of the bicomplex number $\eta$ as $\eta_1 = \eta_{11} + \eta_{12} \im{1}$ and $\eta_2 = \eta_{21} + \eta_{22} \im{1}$, where $\eta_{11}, \eta_{12} , \eta_{21} , \eta_{22} \in \mM \mC (0) = \mR$, we can rewrite $\eta$ as
    $$
        \eta = (\eta_{11} + \eta_{12} \im{1}) + (\eta_{21} + \eta_{22} \im{1}) \im{2} .
    $$
Distributing the number $\im{2}$ and relabeling the index of each real component, we obtain
    $$
        \eta = \eta_{1} + \eta_{\im{1}} \im{1} + \eta_{\im{2}} \im{2} + \eta_{\im{1}\im{2}} \im{1}\im{2} .
    $$
We can therefore write any bicomplex number as a $\mM \mC (1)$-linear (complex linear) combination of the units $1$ and $\im{2}$ or as a $\mM \mC (0)$-linear (real linear) combination of the units $1$, $\im{1}$, $\im{2}$, and $\im{1}\im{2}$. 

The decomposition of a bicomplex number as a real-linear combination of units can be generalized to any multicomplex numbers. Let $\mI (n)$ be the set of all units of the form $\im{1}^{a_1} \im{2}^{a_2} \cdots \im{n}^{a_n}$, where $a_j \in \{ 0 , 1\}$. For instance, we have 
    $$
        \mI (2) = \{ 1 , \im{1} , \im{2} , \im{1}\im{2} \} \quad \text{ and } \quad \mI (3) = \{ 1 , \, \im{1} , \, \im{2} , \, \im{1}\im{2} , \, \im{3} , \, \im{1} \im{3} , \, \im{2} \im{3} , \, \im{1}\im{2}\im{3} \} .
    $$
From combinatorial considerations, we can show that the number of elements of $\mI (n)$ is $2^n$.

Within the set $\mI (n)$, there are units that square to $1$ or to $-1$. An imaginary unit is a unit $\um{}$ having the property that $\ums{} = -1$ and a hyperbolic unit is a unit $\um{}$ that have the property that $\ums{} = 1$ and $\um{} \neq 1$. Combinatorial considerations implies that the number of imaginary units in $\mI (n)$ is $2^{n-1}$ and the number of hyperbolic units in $\mI (n)$ is $2^{n-1} - 1$. For more details on hyperbolic units, the reader is referred to \cite{doyon2022}, where the authors found more imaginary and hyperbolic units in the set of multicomplex numbers. For this paper, however, only the set $\mI (n)$ will be considered. Therefore, we can decompose a multicomplex number $\eta$ as
    $$
        \eta = \sum_{\um{} \in \mI (n)} x_{\um{}} \um{} ,
    $$
where $\eta_{\um{}} \in \mM \mC (0) = \mR$ for any $\um{} \in \mI (n)$. 

We presented the example of a bicomplex numbers. Another important example is the set of tricomplex numbers that we will use later in the paper as our optimal space for the 3D slices of the Filled-in Julia sets. Any tricomplex number $\eta$ can be written down as a real-linear combination of the units from the set $\mI (3)$ as followed:
    $$
        \eta = x_1 + x_{\im{1}} \im{1} + x_{\im{2}} \im{2} + x_{\im{1}\im{2}} \im{1}\im{2} + x_{\im{3}} \im{3} + x_{\im{1}\im{3}} \im{1} \im{3} + x_{\im{2}\im{3}} \im{2}\im{3} + x_{\im{1}\im{2}\im{3}} \im{1}\im{2}\im{3}
    $$
  where $x_{\um{}} \in \mR$ for $\um{} \in \mI (3)$. For reference later in the paper, we set $\jm{1} = \im{1}\im{2}$, $\jm{2} := \im{1}\im{3}$, $\jm{3} := \im{2}\im{3}$, and $\im{4} := \im{1}\im{2}\im{3}$.

\subsection{Idempotent Representations}
Define the multicomplex numbers $\be{n}$ and $\bec{n}$, for $n \geq 2$, as followed:
$$
    \be{n} = \frac{1+\im{n-1}\im{n}}{2} \quad \text{  and  } \quad \bec{n}=\frac{1-\im{n-1}\im{n}}{2}
$$
Given $\eta = \eta_1 + \eta \im{n}$, it is easy to see that
    $$  
        \eta \be{n} = (\eta_1 - \eta_2 \im{n-1}) \be{n} \quad \text{ and } \quad \eta \bec{n} = (\eta_1 + \eta_2 \im{n-1})\bec{n} .
    $$
Since $\be{n} + \bec{n} = 1$, we obtain
    $$  
        \eta (\be{n} + \bec{n}) = \eta \be{n} + \eta \bec{n} = (\eta_1 - \eta_2 \im{n-1} ) \be{n} + (\eta_1 + \eta_2 \im{n-1})\bec{n} .
    $$
This is called the idempotent representation of a multicomplex number. We set $\eta_{\be{}} := \eta_1 - \eta_2 \im{n - 1}$ and $\eta_{\bec{n}} := \eta_1 + \eta_2 \im{n-1}$ and write the idempotent representation of a multicomplex number $\eta$ as followed
    $$
        \eta = \eta_{\be{}} \be{} + \eta_{\bec{}} \bec{} .
    $$ 
Because $\be{n} \bec{n} = 0$, $\be{n}^2 = \be{n}$, and $\bec{n}^2 = \bec{n}$, we obtain the following properties of the idempotent representation. 
\begin{theorem}
 Let $\eta=\eta_{\be{}} \be{} + \eta_{\bec{}}\bec{}$ and $\zeta = \zeta_{\be{}}\be{} + \zeta_{\bec{}} \bec{}$ be two multicomplex numbers and $n \geq 2$ be an integer. Then
    \begin{enumerate}
        \item $\eta + \zeta = (\eta_{\be{}} + \zeta_{\be{}}) \be{} + (\eta_{\bec{}} + \zeta_{\bec{}}) \bec{}$.
        \item $\eta \zeta = (\eta_{\be{}} \zeta_{\be{}}) \be{} + (\eta_{\bec{}} \zeta_{\bec{}}) \bec{}$.
        \item $\eta^m = \eta_{\be{}}^m \be{} + \zeta_{\bec{}}^m \bec{}$, for any integer $m \geq 1$. 
    \end{enumerate}
 \end{theorem}

Using the idempotent representation we can defined an operation similar to the Cartesian product of two sets.
\begin{definition}
 Let $X \subset \mM \mC (n)$, for $n \geq 2$. The set $X$ is called an $\mM \mC (n)$-Cartesian product determined by two sets $X_1 \subset \mM \mC (n - 1)$ and $X_2 \subset \mM \mC (n - 1)$ if the following holds
    $$
        X = X_1 \times_{\be{n}} X_2 = \{ \eta_{\be{}} \be{} + \eta_{\bec{}} \bec{} \, : \, \eta_{\be{}} , \eta_{\bec{}} \in \mM \mC (n - 1) \} .
    $$
\end{definition}

\subsection{Norm}
The norm of a multicomplex number $\eta = \eta_1 + \eta_2 \im{n} \in \mathbb{M}(n)$ is defined recursively as the following quantity:
    $$
        \Vert \eta \Vert_n := \sqrt{\Vert \eta_1 \Vert_{n-1}^2 + \Vert \eta_2 \Vert_{n-1}^2 \big)}
    $$
where $\Vert \eta \Vert_1 := \sqrt{\eta_1^2 + \eta_2^2}$. It can be shown that the expression of the norm of a multicomplex number can be written down in terms of the idempotent components in the idempotent representation of a multicomplex number:
    $$
        \Vert \eta \Vert_n = \sqrt{\frac{\Vert \eta_{\be{}} \Vert_{n-1}^2 + \Vert \eta_{\bec{}} \Vert_{n-1}^2}{2}}
    $$
where $\eta = \eta_1 + \eta_2 \im{n} = \eta_{\be{}} \be{} + \eta_{\bec{}} \bec{}$. 

The topology used on $\mM \mC (n)$ is the one induced by the metric $d (\eta , \zeta ) := \Vert \eta - \zeta \Vert_n$. The open balls are the set $B_n (\eta , r) := \{ \zeta \in \mM \mC (n) \, : \, \Vert \zeta - \eta \Vert_n < r \}$. As a consequence, the $\mM \mC (n)$-cartesian product preserved closedness, compactness, openness, and connectedness of sets, that is
    \begin{enumerate}
        \item If $X_1$ and $X_2$ are closed sets in $\mM \mC (n-1)$, then the set $X_1 \times_{\be{n}} X_2$ is closed in $\mM \mC (n)$.
        \item If $X_1$ and $X_2$ are compact sets in $\mM \mC (n-1)$, then the set $X_1 \times_{\be{n}} X_2$ is compact in $\mM \mC (n)$.
        \item If $X_1$ and $X_2$ are open sets in $\mM \mC (n-1)$, then the set $X_1 \times_{\be{n}} X_2$ is open in $\mM \mC (n)$.
        \item If $X_1$ and $X_2$ are connected sets, then the set $X_1 \times_{\be{n}} X_2$ is a connected set.
    \end{enumerate}

\section{Preliminaries on the Filled-in Julia Sets in the Complex Plane}\label{Sec:ComplexFilledJ}

In this section, we define the Filled-in Julia sets, denoted by $\cK_{1, c}^p$, over the complex plane $\mM \mC (1)$. We will show some basic properties of the set $\cK_{1, c}^p$. These properties will be essential to draw pictures of its generalization to the bicomplex and tricomplex spaces. 

\subsection{Definition}
For a fixed complex number $c \in \mM \mC (1)$, the filled-in Julia set associated to the polynomial function $f_c (z) = z^p + c$ is defined as followed:
    $$
        \cK_{1, c}^p := \{ z \in \mM \mC (1) \, : \, \text{ the sequence } (f_c^n (z))_{n = 1}^\infty \text{ is bounded} \} .
    $$
Here the notation $f_c^n (z)$ stands for the $n$-th iterate of the function $f_c$ at $z$, meaning $f_c^n (z) = f (f^{n - 1} (z))$ when $n \geq 2$ and $f_c (z) = z^p + c$ when $n = 1$. 

\subsection{Basic Properties}
The first important property of the complex Filled-in Julia sets is the fact that they are compact subsets of the complex plane. We can in fact identify exactly the radius of the disk containing a filled-in Julia set.

\begin{theorem}
Let $c \in \mM \mC (1)$. Then $\cK_{1, c}^p$ is a compact set. In particular, we have that $\cK_{1, c}^p \subseteq \overline{B (0, R)}$, where $R := \max \{ |c| , 2^{1/(p-1)} \}$. 
\end{theorem}
\begin{proof}
The fact that $\cK_{1, c}^p$ is a compact subset of the complex plane is a known fact from the theory of holomorphic dynamics. A reference for this result is \cite[Theorem 14.2]{F2014}.

We will prove the second part of the theorem. Let $R$ be as in the statement of the theorem. We will show that $\mM \mC (1) \backslash \overline{B (0, R)} \subset \mM \mC (1) \backslash \cK_{1, c}^p$, which is equivalent to the statement we want to show. 

Assume that $z \not\in \overline{B (0, R)}$. Then $|z| > R$ and 
    $$
    |f_c (z)| = |z^p + c| \geq |z|^p - |c| > |z|^p - |z| = |z| (|z|^{p-1} - 1) .
    $$
where, in the second inequality, we used the fact that $|c| \leq R$. Since $z \not\in \overline{B (0, R)}$, then $|z| = R + \delta$ for some positive real number $\delta$. This implies that
    $$
        |z|^{p-1} - 1 = (R + \delta)^{p-1} - 1 = R^{p-1} (1 + \delta / R)^{p-1} - 1 .
    $$
Recall Bernouilli's inequality $(1 + x)^m \geq 1 + mx$ for any $x \geq 0$ and $m \in \mN$. Set $x = \delta /R$ and $m = p-1 \in \mN$, then we obtain the following inequality:
    $$
        |z|^{p-1} - 1 \geq R^{p-1} (1 + (p-1) \delta / R) - 1 = R^{p-1} + (p-1) \delta R^{p-2} - 1 .
    $$
Using the fact that $R \geq 2^{1/(p-1)}$, we have that $R^{p-1} - 1 \geq 1$ and we obtain the following lowerbound:
    $$
    |z|^{p-1} - 1 \geq 1 + (p-1) \delta R^{p-2} .
    $$
Also, since $p - 2 \geq 0$ and $R > 1$, we have that $R^{p-2} \geq 1$ and this means
    $$
    |z|^{p-1} - 1 \geq 1 + (p-1) \delta. 
    $$
Therefore, we obtain
    $$
    |f_c (z)| > |z| (1 + (p-1) \delta ) .
    $$
Assume that $|f_c^k (z)| = |z| (1 + (p-1) \delta )^k$ for some integer $k$. Then, we have
    $$
        |f_c^{k + 1} (z)| \geq |f_c^k (z)| (1 + (p-1) \delta ) \geq |z| (1 + (p-1) \delta )^k (1 + (p-1) \delta ) = |z| (1 + (p-1) \delta )^{k + 1} .
    $$
Hence, by the Principle of Induction, we just proved that if $|z| > R$, then
    $$
    |f_c^n (z)| \geq |z| (1 + (p-1) \delta )^{n} \quad (\forall n \in \mN ) .
    $$
Since $1 + (p-1) \delta > 1$, then $\lim_{n \ra \infty} |f_c^n (z)| = \infty$ and therefore the sequence $(f_c^n (z))_{n = 1}^\infty$ is unbounded. This means $z \not\in \cK_{1, c}^p$. This concludes the proof. \qed
\end{proof}

From the above proof, we can extract an important lemma that will be at the core of the justification of the algorithm to draw filled-in Julia sets.

\begin{lemma}\label{Lem:LowerBoundForIteration}
    Let $c \in \mM \mC (1)$ and $R := \max \{ |c| , 2^{1/(p-1)} \}$. If there exists a positive integer $m$ such that $|f_c^m (z)| > R$, then there is some positive real number $\delta$ such that $|f_c^{m + n} (z)| \geq |f_c^{m} (z)| (1 + (p-1) \delta )^n$ for any integer $n \geq 1$. 
\end{lemma}
\begin{proof}
Assume that $|f_c^m (z)| > R$ for some positive integer $m$. Let $\delta > 0$ be such that $|f_c^m (z)| = R + \delta$. Reproducing the proof of the preceding Theorem with $z$ replaced by $f_c^m (z)$, we can show that
    $$
        |f_c^{m + 1} (z)| \geq |f_c^m (z)| (1 + (p-1) \delta ) .
    $$
This is the base case $n = 1$. Assuming that $|f_c^{m + k} (z)| \geq |f_c^m (z)| (1 + (p-1) \delta )^k$, we then have
    $$
    |f_c^{m + k + 1} (z)| \geq |f_c^{m + k} (z)| (1 + (p - 1) \delta ) \geq |f_c^m (z)| (1 + (p-1) \delta )^{k + 1}
    $$
which establishes the induction step. The Principle of Induction then concludes the proof. \qed
\end{proof}

\subsection{An Algorithm}
We can now prove the following theorem which gives an algorithm to draw pictures of filled-in Julia sets in the complex plane.

\begin{theorem}\label{Thm:AlgorithmDrawComplexFilledInJuliaSets}
    Let $z, c \in \mM \mC (1)$ and $R := \max \{ |c| , 2^{1/(p-1)} \}$. The following assertions are equivalent:
        \begin{enumerate}
            \item $z \in \cK_{1, c}^p$.
            \item $|f_c^n (z)| \leq R$ for any positive integer $n$.
        \end{enumerate}
\end{theorem}
\begin{proof}
    If $|f_c^n (z)| \leq R$ for any positive integer $n$, then the sequence $(f_c^n (z))_{n = 1}^\infty$ is bounded and hence $z \in \cK_{1, c}^p$. This proves the implication (2) $\Rightarrow$ (1).

    To prove the implication (1) $\Rightarrow$ (2), we will prove the contrapositive. Assume that there is a positive integer $m$ such that $|f_c^m (z)| > R$. By Lemma \ref{Lem:LowerBoundForIteration}, we deduce that for some $\delta > 0$, we have
        $$
        |f_c^{m + n} (z)| \geq |f_c^m (z)| (1 + (p-1) \delta )^n \quad (\forall n \in \mN ) .
        $$
    Since $1 + (p-1) \delta > 1$, then $\lim_{k \ra \infty} |f^k_c (z)| = \infty$ and hence $z \not\in \cK_{1, c}^p$. The contrapositive of the implication (1) $\Rightarrow$ (2) is therefore true and hence the implication is also true. \qed
\end{proof}

From Theorem \ref{Thm:AlgorithmDrawComplexFilledInJuliaSets}, we can construct an algorithm to draw a filled-in Julia set in the complex plane. Fix $c \in \mM \mC (1)$, $R = \max \{ |c| , 2^{1/(p-1)} \}$ and $N \in \mN$. For each complex number in the square $[-R , R] \times [-R , R]$, 
    \begin{enumerate}
        \item Compute the $n$-th iterate $f_c^n (z) = f_c (f_c^{n- 1} (z))$, for $n = 1, 2, \ldots , N$.
        \item At every iteration, compute $W := |f_c^n (z)|$ and verify the following conditions:
            \begin{enumerate}
                  \item If $W > R$ for some $n < N$, then assume that $z \not\in \cK_{1, c}^p$ and assign a color to the number $z$.
                  \item if $W \leq R$ for every $n = 1 , 2 , \ldots , N$, then assume that $z \in \cK_{1, c}^p$ and assign a different color than the one chosen in (a) to the number $z$.
            \end{enumerate}
    \end{enumerate}

\section{Higher Dimensional Filled-in Julia Sets}\label{Sec:MulticomplexFilledJ}

We now generalize the filled-in Julia sets to higher dimensions. This was done previously in \cite{PR2009} for the polynomial $f_c (\zeta ) = \zeta^2 + c$. We generalize it here for the polynomial $f_c (\zeta ) = \zeta^p + c$.

Throughout this section, the function $f_c$ will be the polynomial $f_c (\zeta ) = \zeta^p + c$, where $\zeta , c \in \mM \mC (n)$ and $p, n$ are positive integers greater than or equal to $2$. As we used in the previous section, the $m$-th iterate of the function $f_c$ is defined as $f_c^m (\zeta ) = f (f_c^{m-1} (\zeta ))$ when $m \geq 2$ and $f_c^1 ( \zeta ) = f_c ( \zeta )$ when $m = 1$.

\subsection{Multicomplex Filled-in Julia Sets}

The multicomplex filled-in Julia set associated to the polynomial $f_c$, with $c \in \mM \mC (n)$ is defined as followed.

\begin{definition}
    Let $c \in \mM \mC (n)$. The multicomplex filled-in Julia set associated to $f_c$ is denoted by $\cK_{n, c}^p$ and 
        $$
            \cK_{n, c}^p := \{ \zeta \in \mM \mC (n) \, : \, \text{ the sequence } (f_c^m (\zeta ))_{m = 1}^\infty \text{ is bounded} \} .
        $$
\end{definition}

When $n = 1$, we recover the definition of the filled-in Julia sets in the complex plane. Two other cases are of importance in this paper:
    \begin{enumerate}
        \item When $n = 2$, $\mM \mC (2)$ is the set of bicomplex numbers. We therefore obtain the bicomplex version of the filled-in Julia sets denoted by $\cK_{2, c}^p$ (see \cite{rochon2003} for the case $p = 2$).
        \item When $n = 3$, $\mM \mC (3)$ is the set of tricomplex numbers. We therefore obtain the tricomplex version of the filled-in Julia sets denoted by $\cK_{3, c}^p$ (see \cite{PR2009} for $p = 2$).
    \end{enumerate}
The case of the tricomplex numbers will be particularly important for our main result.

\subsection{Basic properties}

Some of the basic properties stated in this section were already mentioned, without proofs, in \cite{BPR2019}. For the sake of the reader, we provide the complete proofs of these results here.

Using the idempotent representation of a multicomplex number, we can obtain the following important decomposition of a multicomplex filled-in Julia set.

\begin{theorem}\label{Thm:CartesianProductFilledJulia}
    Let $n \geq 2$ be an integer and $c = c_{\be{n}} \be{n} + c_{\bec{n}} \bec{n} \in \mM \mC (n)$, with $c_{\be{n}} , c_{\bec{n}} \in \mM \mC (n - 1)$ the idempotent components of $c$. Then we have
        $$
            \cK_{n, c}^p = \cK_{n - 1 , c_{\be{n}}}^p \times_{\be{n}} \cK_{n - 1, c_{\bec{n}}}^p .
        $$
\end{theorem}
\begin{proof}
    Let $n$ and $c$ be as in the statement of the theorem. Using the idempotent representation, for any $m \geq 1$ and $\zeta = \zeta_{\be{n}}\bec{n} + \zeta_{\bec{n}} \bec{n}$, we have
        $$  
            f_c^m (\zeta ) = f_{c_{\be{n}}}^m (\zeta_{\be{n}}) \be{n} + f_{c_{\bec{n}}}^m (\zeta_{\bec{n}}) \bec{n} 
        $$
    and therefore
        $$
            \Vert f_c^m (\zeta ) \Vert_n = \sqrt{\frac{\Vert f_{c_{\be{n}}}^m (\zeta_{\be{n}}) \Vert_{n-1}^2 + \Vert f_{c_{\bec{n}}}^m (\zeta_{\bec{n}}) \Vert_{n-1}^2}{2}} .
        $$
    It immediately follows from this last equality that the sequence $(f_c^m (\zeta ))_{m = 1}^\infty$ is bounded if and only if the two sequences $(f_{c_{\be{n}}}^m (\zeta_{\be{n}}))_{m=1}^\infty$ and $(f_{c_{\bec{n}}}^m (\zeta_{\bec{n}}))_{m=1}^\infty$ are bounded. Hence, $\zeta \in \cK_{n, c}^p$ if and only if $\zeta_{\be{n}} \in \cK_{n-1, c_{\be{n}}}^p$ and $\zeta_{\bec{n}} \in \cK_{n-1, c_{\bec{n}}}^p$ if and only if $\zeta \in \cK_{n-1, c_{\be{n}}}^p \times_{\be{n}} \cK_{n-1, c_{\bec{n}}}^p$. This completes the proof. \qed
\end{proof}
{\noindent}We collect below two subcases of the previous result. 
    \begin{enumerate}
    \item When $n = 2$, we have $\cK_{2, c}^p = \cK_{1, c_{\be{2}}}^p \times_{\gm{1}} \cK_{1, c_\bec{2}}$, where $c = c_{\be{2}} \be{2} + c_{\bec{2}} \bec{2}$.
        \item When $n = 3$, we obtain $\cK_{3, c}^p = \cK_{2, c_{\be{3}}}^p \times_{\be{3}} \cK_{2, c_\bec{3}}^p$, where $c = c_{\be{3}} \gm{3} + c_{\bec{3}} \gmc{3}$.
    \end{enumerate}

The next result will be important in the next section to draw pictures of the higher dimensional filled-in Julia sets.

\begin{theorem}\label{Thm:MulticomplexFilledInAlgorithm}
    Let $n \geq 1$ be an integer and $c \in \mM \mC (n)$. Let $R := \max \{ |c| , 2^{1/(p-1)} \}$. Then the following assertions are equivalent:
        \begin{enumerate}
            \item $\zeta \in \cK_{n, c}^p$.
            \item $\Vert f_c^m (\zeta ) \Vert_n \leq R$ for every integer $m \geq 1$. 
        \end{enumerate}
\end{theorem}
\begin{proof}
    The proof will be by induction on $n \geq 1$. The case $n = 1$ was proved in the previous section and is Theorem \ref{Thm:AlgorithmDrawComplexFilledInJuliaSets}.

    Now, assume that the statement of the theorem is true for $n = k$, that is, for any $c \in \mM \mC (k)$, (1) is equivalent to (2). We have to prove that (1) is equivalent to (2) for $n = k + 1$. Let $\zeta , c \in \mM \mC (k + 1)$ be expressed in their idempotent representation:
        $$
            c = c_{\be{k+1}} \be{k+1} + c_{\bec{k+1}} \bec{k+1} \quad \text{ and } \quad \zeta = \zeta_{\be{k+1}} \be{k+1} + \zeta_{\bec{k+1}} \bec{k+1}  ,
        $$
    where $c_{\be{k+1}} , c_{\bec{k+1}}, \zeta_{\be{k+1}} , \zeta_{\bec{k+1}} \in \mM \mC (k)$. To simplify the notations, we set $c_1 := c_{\be{k+1}}$, $c_2 := c_{\bec{k+1}}$, $\zeta_1 := \zeta_{\be{k+1}}$, and $\zeta_2 := \zeta_{\bec{k+1}}$.
    
    Assume that $\Vert f_c^m (\zeta ) \Vert_n \leq R$ for any integer $m \geq 1$. Then, from the definition of the filled-in Julia set, $\zeta \in \cK_{k+1, c}^p$. So, (2) implies (1).

    Assume that $\zeta \in \cK_{n, c}^p$. Then, $\zeta_{1} \in \cK_{k, c_{1}}^p$ and $ \zeta_{2} \in \cK_{k, c_{2}}^p$ by Theorem \ref{Thm:CartesianProductFilledJulia}. By the induction hypothesis, we get
        $$
            \Vert f_{c_1}^m (\zeta_1) \Vert_k \leq R \quad \text{ and } \quad \Vert f_{c_2}^m (\zeta_2) \Vert_k \leq R
        $$
    for any integer $m \geq 1$. Using the fact that for any integer $m \geq 1$, we have
        $$
            f_{c}^m (\zeta) = f_{c_1}^m (\zeta_1) \be{k + 1} + f_{c_2}^m (\zeta_2) \bec{k + 1} ,
        $$
    and using the expression of the norm in terms of the idempotent components, we get
        $$
            \Vert f_c^m (\zeta ) \Vert_{k + 1} = \sqrt{\frac{\Vert f_{c_1}^m (\zeta_1) \Vert_k^2 + \Vert f_{c_2}^m (\zeta_2) \Vert_k^2}{2}} \leq \sqrt{R^2} = R .
        $$
    Hence $\Vert f_c^m (\zeta ) \Vert_{k + 1} \leq R$ and the statement is proved for $n = k + 1$. 

    By the Principle of Mathematical Induction, the proof is complete. \qed 
\end{proof}

\section{Visualization through 3D Slices}\label{Sec:Visualisation3D}
In this section, we will be interested in visualizing the higher dimensional version of the filled-in Julia sets over the multicomplex numbers. Thoughout this section, the letter $n$ is used to denote a positive integer greater than or equal to $2$. We also let $p$ be a positive integer greater than or equal to $2$.

In order to visualize the multicomplex filled-in Julia sets, we define the concept of a 3D slice. We will first need to define (real) vector subspaces of $\mM \mC (n)$.

\begin{definition}
    Let $\um{1}, \um{2} , \um{3} \in \mI (n)$ be distinct. We define the following (real) vector subspace of $\mM \mC (n )$:
        $$
        \mT (\um{1}, \um{2} , \um{3}) := \spn_\mR \{ \um{1} , \um{2} , \um{3} \} := \{ x \um{1} + y \um{2} + z \um{3} \, : \, x, y, z \in \mR \} .
        $$
\end{definition}
{\noindent}For instance, if $n = 3$, then we have $\mI (3) = \{ 1, \im{1} , \im{2} , \jm{2} , \im{3} , \jm{2} , \jm{3} , \im{4} \}$. The principal units $\um{1} , \um{2} , \um{3}$ can therefore be real, imaginary, or hyperbolic in nature. We can now define a principal 3D slices of a multicomplex filled-in Julia set.

\begin{definition}
    Let $c \in \mM \mC (n)$ and let $\um{1} , \um{2} , \um{3} \in \mI (n)$ be distinct. A principal 3D slice of the filled-in Julia set $\cK_{n, c}^p$ is defined as followed:
        $$
            \cT^p (\um{1} , \um{2} , \um{3} ) := \{ \zeta \in \mT (\um{1} , \um{2} , \um{3}) \, : \, \text{the sequence } (f_c^m (\zeta ))_{m = 1}^\infty \text{ is bounded} \} .
        $$
\end{definition}

From Theorem \ref{Thm:MulticomplexFilledInAlgorithm}, we can construct an algorithm to visualize multicomplex filled-in Julia sets in 3D for $n \geq 2$. Fix $c \in \mM \mC (n)$, $R = \max \{ \Vert c \Vert_n , 2^{1/(p-1)} \}$ and $N \in \mN$. Fix also three distinct principal units $\um{1}, \um{2}, \um{3} \in \mI (n)$. For each multicomplex number $\zeta = \zeta_1 \um{1} + \zeta_2 \um{2} + \zeta_3 \um{3}$, where $\zeta_1, \zeta_2, \zeta_3 \in [-R , R] $, 
    \begin{enumerate}
        \item Compute the $m$-th iterate $f_c^m (\zeta ) = f_c (f_c^{m- 1} (\zeta ))$, for $m = 1, 2, \ldots , M$.
        \item At every iteration, compute $W := \Vert f_c^m (\zeta )\Vert_n$ and verify the following conditions:
            \begin{enumerate}
                  \item If $W > R$ for some $m < M$, then assume that $\zeta \not\in \cK_{n, c}^p$ and assign a color to the number $\zeta$.
                  \item if $W \leq R$ for every $m = 1 , 2 , \ldots , M$, then assume that $\zeta \in \cK_{n, c}^p$ and assign a different color than the one chosen in (a) to the number $\zeta$.
            \end{enumerate}
    \end{enumerate}
See Figure \ref{Fig:3DPSlices-Examples} for two examples of principal 3D slices visualized using the above algorithm. More powerful methods exist to visualize principal 3D slices, see \cite{Pierre_Guillaume_Dominic_2019}.

\begin{figure}
    \centering
    \begin{subfigure}{0.4\textwidth}
        \includegraphics[width=\textwidth]{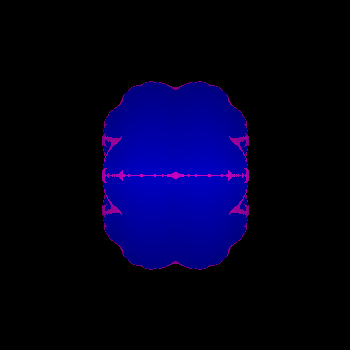}
        \caption{$\cT^2(1,\im{1},\im{2})$}
        \label{fig:octahedron}
    \end{subfigure}
    \begin{subfigure}{0.4\textwidth}
        \includegraphics[width=\textwidth]{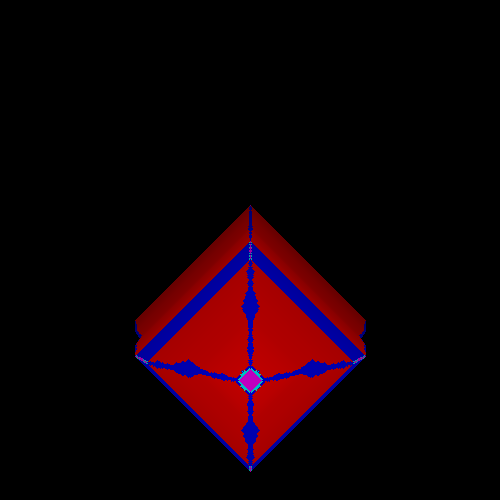}
        \caption{$\cT^3(\jm{1},\jm{3},\im{4})$}
        \label{fig:TCJulia_j1-j3-i4.png}
    \end{subfigure}
    \caption{Graphical results of principal 3D slices of the filled-in Julia set obtained using the algorithm}\label{Fig:3DPSlices-Examples}
\end{figure}

\section{Definitions and Preliminaries for the Characterization of 3D slices}\label{Sec:DefinitionEquivalence3D}

In this section, we present the main tools needed to obtain our main result on the characterization of the principal 3D slices. 

\subsection{An equivalent relation}
The next definition introduces the space of iterates. 
    \begin{definition}
    Let $c \in \mM (n )$ and $\um{1} , \um{2} , \um{3} \in \mI (n)$. The (real) vector subspace of iterates is defined as followed:
    $$
        \mathrm{It}^p (\um{1} , \um{2} , \um{3} ) := \spn_{\mR} \{ f_c^m (\zeta ) \, : \, \zeta \in \mT (\um{1} , \um{2} , \um{3} ) \text{ and } m \in \mN \} .
    $$
    \end{definition}
When we write "characterization of the 3D slices", what we mean is to find the equivalent classes of the following equivalent relation.

\begin{definition}\label{Def:EquivalentRelation3Dslices}
    Let $c \in \mM (n)$. Let $\um{1} , \um{2} , \um{3} \in \mI (n)$ be distinct principal units and let $\vm{1} , \vm{2} , \vm{3} \in \mI (n)$ be three other distinct principal units\footnote{The units $\vm{1}$, $\vm{2}$, $\vm{3}$ might be the same as the units $\um{1}$, $\um{2}$, $\um{3}$. The only important fact is that you choose a set of three different units.}. Define $M_1 := \mathrm{It}^p (\um{1} , \um{2} , \um{3})$ and $M_2 := \mathrm{It}^p (\vm{1} , \vm{2} , \vm{3} )$. The 3D slices $\cT^p (\um{1} , \um{2} , \um{3})$ and $\cT^p (\vm{1} , \vm{2} , \vm{3})$ are \textit{equivalent}, denoted by $\cT^p (\um{1} , \um{2} , \um{3}) \sim \cT^p (\vm{1} , \vm{2} , \vm{3})$ if and only if there exists a linear bijective application $\varphi : M_1 \ra M_2$ such that 
    \begin{enumerate}
        \item\label{Eq:ConditionEquivalent3DSlicesFirstIterates} $\varphi (f_c (x \um{1} + y \um{2} + z \um{3})) = f_c (x \vm{1} + y \vm{2} + z \vm{3})$, for any $x, y, z \in \mR$.
        \item\label{Eq:ConditionEquivalent3DSlices} $\varphi (f_c (\zeta )) = f_c (\varphi (\zeta ))$, for any $\zeta \in  \mathrm{It}^p (\um{1} , \um{2} , \um{3} )$. 
    \end{enumerate}
\end{definition}

Condition \eqref{Eq:ConditionEquivalent3DSlicesFirstIterates} guarantees the first iterates with a starting point in $\mT (\um{1}, \um{2}, \um{3})$ and in $\mT (\vm{1}, \vm{2}, \vm{3})$ respectively are the same. This is especially important if $\mT (\um{1}, \um{2}, \um{3}) \not\subseteq \mathrm{It}^p (\um{1}, \um{2} , \um{3})$ or $\mT (\vm{1}, \vm{2}, \vm{3}) \not\subseteq \mathrm{It}^p (\vm{1}, \vm{2} , \vm{3})$.

Condition \eqref{Eq:ConditionEquivalent3DSlices} guarantees the subsequent iterates behave in the same way. Indeed, an induction argument shows that
    $$
        \varphi (f_c^m (\zeta )) = f_c^m (\varphi (\zeta )) .
    $$
Since $\varphi$ is a bijective linear map, the inverse is also a bijective map and $\varphi$ and $\varphi^{-1}$ are continuous. Therefore, for $\zeta \in \mT (\um{1}, \um{2}, \um{3})$ and $\eta \in \mT (\vm{1}, \vm{2}, \vm{3})$, the sequence $(f_c^m (\zeta ))_{m = 1}^{\infty}$ is bounded if and only if the sequence $(f_c^m (\eta ))_{m = 1}^{\infty}$ is bounded. 

Notice that Definition \ref{Def:EquivalentRelation3Dslices} is similar to Definition 5 in \cite{Brouillette_2019}. However in general the map $\varphi$ can't be applied directly to $\cT^p (\um{1}, \um{2}, \um{3})$ to obtain $\cT^p (\vm{1}, \vm{2}, \vm{3})$ because it might not be defined on the subspaces $\mT (\um{1}, \um{2}, \um{3})$ and $\mT (\vm{1}, \vm{2}, \vm{3})$.  In opposition to what was just said, the map $\varphi$ in the definition in \cite{Brouillette_2019} is always defined on the subspaces $\mT (\um{1}, \um{2}, \um{3})$ and $\mT (\vm{1}, \vm{2}, \vm{3})$. Therefore, our definition covers both cases, the multicomplex filled-in Julia sets and the multicomplex Mandelbrot sets.

\begin{theorem}
    Let $c \in \mM (n)$. The relation $\sim$ is an equivalent relation on the set of 3D slices of $\cK_{n, c}^p$, meaning that $\sim$ is a reflexive, symmetric, and transitive binary operation on the family of principal 3D slices.
\end{theorem}
\begin{proof}
We have to show that the binary operation $\sim$ is reflexive, symmetric, and transitive.
    \begin{enumerate}[label=(\alph*)]
        \item Let $\cT^p(\um{1},\um{2},\um{3})$ be a principal 3D slice. Let $M_1 = M_2 = \mathrm{It}^p (\um{1}, \um{2}, \um{3})$ and let $\varphi : M_1 \ra M_2$ be the identity map $\varphi (\zeta ) = \zeta$. Then, all conditions in the definition of the relation $\sim$ are automatically verified and $\cT^p (\um{1}, \um{2}, \um{3}) \sim \cT^p (\um{1}, \um{2} , \um{3})$.
        \item Let $\cT^p (\um{1}, \um{2}, \um{3})$ and $\cT^p (\vm{1}, \vm{2}, \vm{3})$ be two principal 3D slices such that 
            $$
                \cT^p (\um{1}, \um{2}, \um{3}) \sim \cT^p (\vm{1}, \vm{2}, \vm{3}) .
            $$ 
        Then there is a bijective map $\varphi : M_1 \ra M_2$, where $M_1 = \mathrm{It}^p (\um{1}, \um{2}, \um{3})$ and $M_2 = \mathrm{It}^p (\vm{1}, \vm{2} , \vm{3})$, such that
            \begin{enumerate}[label=(\arabic*)]
                \item $\varphi (f_c (x \um{1} + y \um{2} + z \um{3})) = f_c (x \vm{1} + y \vm{2} + z \vm{3})$, for any $x, y, z \in \mR$.
                \item$\varphi (f_c (\zeta )) = f_c (\varphi (\zeta ))$, for any $\zeta \in  \mathrm{It}^p (\um{1} , \um{2} , \um{3} )$. 
            \end{enumerate}
        Let $\psi := \varphi^{-1}$. Therefore, $\psi : M_2 \ra M_1$ and from the condition (1), we get
            \begin{align*} 
                \psi (f_c (x \vm{1} + y \vm{2} + z \vm{3})) &= \varphi^{-1}\big(\varphi (f_c (x \um{1} + y \um{2} + z \um{3}))\big) \\ 
                &= f_c (x \um{1} + y \um{2} + z \um{3}) .
            \end{align*}
        Also, writing $\zeta = \varphi (\psi (\zeta ))$, from condition (2), we obtain
            \begin{align*}
                f_c^m (\zeta ) = f_c^m (\varphi (\psi (\zeta ))) = \varphi (f_c^m (\psi (\zeta )) 
            \end{align*}
        and taking the inverse of $\varphi$ on each side, we get
            $$
                \psi (f_c^m (\zeta )) = f_c^m (\psi (\zeta )) .
            $$
        Hence, $\phi$ satisfies the two conditions in Definition \ref{Def:EquivalentRelation3Dslices} and $\cT^p (\vm{1}, \vm{2}, \vm{3}) \sim \cT^p (\um{1}, \um{2}, \um{3})$. 
        \item To simplify the notation, we will omit the units in the 3D slices. Let $\cT_1^p$, $\cT_2^p$, and $\cT_3^p$ be three principal 3D slices such that $\cT_1^p \sim \cT_2^p$ and $\cT_2^p \sim \cT_3^p$. Then there exist two bijective maps $\varphi_1 : M_1 \ra M_2$ and $\varphi_2 : M_2 \ra M_3$ such that conditions (1) and (2) from Definition \ref{Def:EquivalentRelation3Dslices} are satisfied. Then it is not hard to see that the map $\psi : M_1 \ra M_3$ defined by $\psi (\zeta ) = \varphi_2 (\varphi_1 (\zeta))$ satisfies conditions (1) and (2) of Definition \ref{Def:EquivalentRelation3Dslices}. Hence, $\cT_1^p \sim \cT_3^p$ and this ends the proof. \qed
    \end{enumerate}    
\end{proof}

\subsection{Characterization of the space of iterates}

The goal now is to find simple expressions for the space of iterates. To do so, we introduce the following (real) vector subspaces.
    \begin{enumerate}
        \item $\mM (\um{1} , \um{2} , \um{3}) := \spn_{\mR} \{ \um{1} , \um{2} , \um{3} , \um{1}\um{2}\um{3} \}$.
        \item $\mL (\um{1} , \um{2} , \um{3} ) := \spn_{\mR} \{ 1 , \um{1}\um{2} , \um{1}\um{3} , \um{2} \um{3} \}$.
        \item $\mS (\um{1} , \um{2} , \um{3}) := \spn_{\mR} \{ 1 , \um{1} , \um{2} , \um{3} , \um{1}\um{2} , \um{1}\um{3} , \um{2}\um{3} , \um{1}\um{2}\um{3}\}$.
    \end{enumerate}
If one of the $\um{1}$, $\um{2}$, $\um{3}$ is equal to $1$, then we will always assume, without loss of generality, that $\um{1} = 1$. We can now prove the following results.
\begin{theorem}\label{Thm:ClosureOfSubAlgebra}
    Let $\um{1} , \um{2} , \um{3} \in \mI (n)$. Then
        \begin{enumerate}
            \item\label{it:ClosureofS} The space $\mS (\um{1} , \um{2} , \um{3})$ is closed under multiplication.
            \item\label{it:EqualityBwMandL} If $\um{1} = 1$ or $\um{i} \um{j} = \pm \um{l}$ for some index $i, j, l \in \{ 1 , 2, 3 \}$, then $\mM (\um{1} , \um{2} , \um{3}) = \mL (\um{1} , \um{2} , \um{3})$.
            \item\label{it:LclosedUnderMultiplication} The space $\mL (\um{1} , \um{2} , \um{3} )$ is closed under multiplication.
            \item\label{it:M-absorb-L} If $\eta \in \mM (\um{1} , \um{2} , \um{3})$ and $\zeta \in \mL (\um{1} , \um{2} , \um{3} )$, then $\eta \zeta \in \mM (\um{1} , \um{2} , \um{3} )$. 
            \item\label{it:M-times-M-equals-L} If $\eta , \zeta \in \mM (\um{1} , \um{2} , \um{3})$, then $\eta \zeta \in \mL (\um{1} , \um{2} , \um{3} )$. 
            \item\label{it:T-subspaceOfL} If $u_1 = 1$ or $\um{k} \um{l} = \pm \um{h}$ for some index $k, l, h \in \{ 1, 2, 3 \}$, then $\mT (\um{1}, \um{2}, \um{3}) \subseteq \mL (\um{1}, \um{2}, \um{3})$.
        \end{enumerate}
\end{theorem}
\begin{proof}
    See \cite[Lemma 2]{Brouillette_2019} for the complete proof of part (1).

    To show (2), assume first that $\um{1} = 1$. From the definition of both subspaces, we get the equality $\mM (1, \um{2} , \um{3}) = \mL (1 , \um{2} , \um{3})$. Now, assume, without loss of generality, that $\um{2} \um{3} = \pm \um{1}$. From that additional assumption, we see that $\um{1}\um{2} \um{3} = \pm 1 = d (1)$ with $d = \pm 1$, $\um{1} \um{2} = \pm \ums{2} \um{3} = a \um{3}$ with $a = \pm \ums{2} \neq 0$, $\um{1} \um{3} = \pm \ums{3} \um{2} = b \um{2}$ with $b = \pm \ums{3} \neq 0$ and $\um{2} \um{3} = \pm \um{1} = c \um{1}$ with $c = \pm 1 \neq 0$. Therefore, we get
        $$
            \eta = A \um{1} + B \um{2} + C \um{3} + D \um{1} \um{2} \um{3} = D d + Ac \um{2} \um{3} + Bb \um{1} \um{3} + Ca \um{1} \um{2}
        $$
    and hence $\eta \in \mM (\um{1} , \um{2} , \um{3})$ if and only if $\eta \in \mL (\um{1} , \um{2} , \um{3})$. 

    To show (3), it is sufficient to show that every multiplication of the basis elements generating $\mL (\um{1} , \um{2} , \um{3} )$ is again one of the basis elements. This can be achieved by constructing the following multiplication table:
    \begin{center}
        \begin{tabular}{c||c|c|c|c}
            $\cdot$ & $1$ & $\um{1}\um{2}$ & $\um{1} \um{3}$ & $\um{2}\um{3}$ \\\hline\hline
            $1$ & $1$ & $\um{1}\um{2}$ & $\um{1} \um{3}$ & $\um{2}\um{3}$ \\\hline
            $\um{1}\um{2}$ & $\um{1} \um{2}$ & $\pm 1$ & $\pm \um{2}\um{3}$ & $\pm \um{1} \um{3}$ \\\hline
            $\um{1} \um{3}$ & $\um{1}\um{3}$ & $\pm \um{2}\um{3}$ & $\pm 1$ & $\pm \um{1}\um{2}$ \\\hline
            $\um{2} \um{3}$ & $\um{2}\um{3}$ & $\pm \um{1}\um{3}$ & $\pm \um{1}\um{2}$ & $\pm 1$  
        \end{tabular}
    \end{center}
    Hence, $\mL (\um{1}, \um{2} , \um{3})$ is closed under multiplication.

    To show parts (4) and (5), we use the same strategy that was used to show (3). We therefore omit their proofs and this concludes the proof of the statement. \qed
\end{proof}

We can now prove the following result regarding the precise expression of the vector subspaces $\mathrm{It}^p (\um{1}, \um{2}, \um{3})$. 

\begin{theorem}\label{Thm:CharacterizationIterates}
    Let $\um{1} , \um{2} , \um{3} \in \mI (n)$ and $c \in \mathbb{R}$. Then
        \begin{enumerate}
            \item\label{Thm:CharacterizationIteratesEvenCase} If $p$ is even, then $\mathrm{It}^p (\um{1} , \um{2} , \um{3}) = \mL (\um{1} , \um{2} , \um{3} )$.
            \item\label{Thm:CharacterizationIteratesOddCZero} If $p$ is odd and if $c = 0$, then $\mathrm{It}^p (\um{1} , \um{2} , \um{3} ) = \mM (\um{1} , \um{2} , \um{3} )$. 
            \item If $p$ is odd and if $c \neq 0$, then 
                \begin{enumerate} 
                \item\label{Thm:CharacterizationIteratesOddCNotZeroClosed} $\mathrm{It}^p (\um{1}, \um{2}, \um{3}) = \mM (\um{1}, \um{2}, \um{3} )$ if $\um{1} = 1$ or if $\um{k}\um{l} = \pm \um{h}$ for some distinct $k, l, h \in \{ 1, 2, 3\}$.
                \item\label{Thm:CharacterizationIteratesOddCNotZeroNotClosed} $\mathrm{It}^p (\um{1} , \um{2} , \um{3} ) = \mS (\um{1} , \um{2} , \um{3} )$ otherwise.
                \end{enumerate}
        \end{enumerate}
\end{theorem}
\begin{proof}
We first start by showing part (1). Assume that $p$ is even, so that $p = 2k$ for some positive integer $k$, and that $c \in \mR$. We have to show an equality between two sets. 

We start by showing the inclusion $\mathrm{It}^p (\um{1} , \um{2} , \um{3}) \subseteq \mL (\um{1} , \um{2} , \um{3})$. For the rest of this part, let $\zeta \in \mT(\um{1},\um{2},\um{3})$ with $\zeta = A \um{1} + B \um{2} + C \um{3}$ where $A, B, C \in \mR$. We will prove using induction that $f_c^m (\zeta ) \in \mL (\um{1} , \um{2} , \um{3} )$ for any integer $m \geq 1$. The first iterate is $f^1_c(\zeta) = \zeta^{2k} + c$. We will prove with an induction on $k$ that $\zeta^{2k} \in \mL (\um{1} , \um{2} , \um{3} )$. Starting with $k=1$, we have 
    $$
        \zeta^{2} = A^2\ums{1} + B^2 \ums{2} +  C^2 \ums{3} + 2AB \um{1}\um{2} + 2AC \um{1} \um{3} + 2BC\um{2}\um{3} .
    $$
Since $\ums{1} = \pm 1$, $\ums{2} = \pm 1$, and $\ums{3} = \pm 1$, we deduce that $\zeta^2 \in \mL(\um{1},\um{2},\um{3})$. Now assume for some $k \in \mN$, we have $\zeta^{2k} \in \mL (\um{1} , \um{2} , \um{3} )$. Then, we have
    $$
        \zeta^{2(k+1)} = \zeta^2\zeta^{2k} .
    $$
From Theorem \ref{Thm:ClosureOfSubAlgebra} part (\ref{it:LclosedUnderMultiplication}), the space $\mL (\um{1} , \um{2} , \um{3} )$ is closed under multicomplex multiplication. We know that $\zeta^2 \in \mL (\um{1} , \um{2} , \um{3})$ and, from the induction hypothesis, we have $\zeta^{2k} \in \mL (\um{1} , \um{2} , \um{3})$. Therefore $\zeta^2 \zeta^{2k} \in \mL (\um{1} , \um{2} , \um{3})$ and hence $\zeta^{2(k + 1)} \in \mL (\um{1} , \um{2} , \um{3} )$. Adding the number $c$, we have just shown that $f_c^1 (\zeta ) \in \mL (\um{1} , \um{2} , \um{3} )$. 

Now assume that $f_c^m (\zeta ) \in \mL (\um{1} , \um{2} , \um{3})$ for some integer $m \geq 1$. We want to show that $f_c^{m + 1} (\zeta ) \in \mL (\um{1} , \um{2} , \um{3} )$. The expression of the $(m + 1)$-th iterate is
    $$
        f_c^{m + 1} (\zeta ) = (f_c^m (\zeta ))^{2k} + c .
    $$
From Theorem \ref{Thm:ClosureOfSubAlgebra} part (\ref{it:LclosedUnderMultiplication}), the space $\mL (\um{1} , \um{2} , \um{3})$ is closed under multicomplex multiplication. Hence, we must have that $(f_c^m (\zeta ))^{2k} \in \mL (\um{1} , \um{2} , \um{3})$ because $f_c^m (\zeta) \in \mL (\um{1} , \um{2} , \um{3})$ from the induction hypothesis. Hence, $f_c^{m + 1} (\zeta ) \in \mL (\um{1} , \um{2} , \um{3})$. This concludes the induction on $m$ and the proof of $\mathrm{It}^p (\um{1} , \um{2} , \um{3}) \subseteq \mL (\um{1} , \um{2} , \um{3})$. 

We now want to prove the other direction, that is $\mL (\um{1} , \um{2} , \um{3}) \subset \mathrm{It}^p (\um{1} , \um{2} , \um{3})$. To do that, we will show that each basis elements ($1$, $\um{1}\um{2}$, $\um{1}\um{3}$, $\um{2}\um{3}$) is a linear combination of some iterates. Setting $\zeta = \um{1}$ implies that
    $$
        f_c^1 (\um{1}) = (\um{1})^{2k} + c = (\ums{1})^k  + c = a \cdot 1 ,
    $$
where $a = (\ums{1})^k + c$. Hence, $1 = f_c^1 (\um{1}) / a$. To show that the other basis elements are a linear combination of some iterates, we will use a similar trick from Brouillette and Rochon's paper (see the proof of \cite[Lemma 3]{Brouillette_2019}). For $l, h \in \{ 1, 2, 3 \}$ with $l \leq h$ and $t \in \mR \backslash \{ 0 \}$, we have
    $$
        (t\um{l} + \um{h})^{2k} = \sum_{s = 0}^{2k} \binom{2k}{s} t^s(\um{l})^s (\um{h})^{2k - s} 
    $$
and splitting the sum over the $s$ that are even and the $s$ that are odd, we get
    $$
        (t\um{l} + \um{h})^{2k} = \sum_{s = 0}^k \binom{2k}{2s}t^{2s} (\ums{l})^s (\ums{h})^{k - s} + \Big( \sum_{s = 0}^{k - 1} \binom{2k}{2s + 1} t^{2s + 1} (\ums{l})^s (\ums{h})^{k-s} \Big) \um{l} \um{h} .
    $$
The last equation can be rewritten as $(t \um{l} + \um{h})^{2k} = A (t) + B (t) \um{l}\um{h}$, where $A(t)$ and $B(t)$ are real-valued polynomials of degree $2k$ and $2k - 1$ respectively. Since a polynomial has finitely many roots, we may choose a number $t_0$ (that depends on $l$ and $h$) such that $B(t_0) \neq 0$. We then obtain
    $$
        f_c^1 (t_0 \um{l} + \um{h}) = A(t_0) + c + B(t_0) \um{l}\um{h}
    $$
and using the fact that $1 = f_c^1 (\um{1}) / a$, we can rewrite the last expression as followed
    $$
        \um{l}\um{h} = \frac{f_c^1 (t_0 \um{l} + \um{h})}{B (t_0)}- \Big( \frac{A(t_0) + c}{aB(t_0)} \Big) f_c^1 (\um{1})
    $$
Therefore, we obtain $\mL (\um{1} , \um{2} , \um{3}) \subseteq \mathrm{It}^p (\um{1} , \um{2} , \um{3})$, this finishes the proof of part (1). 

Now we show part (2). Our goal is to show that for $p$ odd and $c = 0$, we have $\mathrm{It}^p(\um{1},\um{2},\um{3}) = \mM(\um{1},\um{2},\um{3})$. We first show that $\mathrm{It}^p(\um{1},\um{2},\um{3}) \subseteq\mM(\um{1},\um{2},\um{3})$. Consider $\zeta \in \mT(\um{1},\um{2},\um{3})$ and $p = 2k+1$ for $k \in \mN$. We will prove by induction on $m$ that $f_c^m (\zeta ) \in \mM (\um{1} , \um{2} , \um{3} )$. For $m= 1$, we have
    $$
        f_c^1 (\zeta ) = \zeta^{2k + 1} .
    $$
From part (1), we know that $\zeta^{2k} \in \mL (\um{1} , \um{2} , \um{3})$ for any $k \geq 1$. Since $\zeta \in \mT (\um{1} , \um{2} , \um{3}) \subseteq \mM (\um{1} , \um{2} , \um{3} )$, it follows from Theorem \ref{Thm:ClosureOfSubAlgebra} part \eqref{it:M-absorb-L} that 
    $$
        \zeta^{2k + 1} = \zeta \zeta^{2k} \in \mM (\um{1} , \um{2} , \um{3}) .
    $$
Hence, $f_c^1 (\zeta ) \in \mM (\um{1} , \um{2} , \um{3} )$. Now assume that $f_c^m (\zeta ) \in \mM (\um{1} , \um{2} , \um{3})$ for some positive integer $m \geq 1$. Then we have
    \begin{align*}
        f_c^{m + 1} (\zeta ) = (f_c^m (\zeta ))^{2k + 1} &= (f_c^m (\zeta ))^{2k} (f_c^m (\zeta ) ) = ( (f_c^m (\zeta ))^{2})^k (f_c^m (\zeta )) .
    \end{align*}
Using Theorem \ref{Thm:ClosureOfSubAlgebra} part \eqref{it:M-times-M-equals-L}, we have $(f_c^m (\zeta ))^2 \in \mL (\um{1} , \um{2} , \um{3})$ since, from the induction hypothesis, $f_c^m (\zeta ) \in \mM (\um{1} , \um{2} , \um{3})$. From Theorem \ref{Thm:ClosureOfSubAlgebra} part \eqref{it:LclosedUnderMultiplication}, the set $\mL (\um{1} , \um{2} , \um{3})$ is closed under multicomplex multiplication and therefore $((f_c^m (\zeta ))^2)^k \in \mL (\um{1} , \um{2} , \um{3})$. Now, using Theorem \ref{Thm:ClosureOfSubAlgebra} part \eqref{it:M-absorb-L}, we obtain that $( (f_c^m (\zeta))^2)^k (f_c^m (\zeta )) \in \mM (\um{1} , \um{2} , \um{3})$ because $f_c^m (\zeta ) \in \mM (\um{1} , \um{2} , \um{3})$. Hence, we obtain $f_c^{m + 1} (\zeta ) \in \mM (\um{1} , \um{2} , \um{3})$. Hence, by the Principle of Mathematical Induction, the proof of the inclusion $\mathrm{It}^p (\um{1} , \um{2} , \um{3}) \subseteq \mM (\um{1} , \um{2} , \um{3} )$ is completed. 

We now want to show that $\mM (\um{1} , \um{2} , \um{3}) \subseteq \mathrm{It}^p (\um{1} , \um{2} , \um{3})$. We will do this by expressing each basis element $\um{1}$, $\um{2}$, $\um{3}$, and $\um{1}\um{2}\um{3}$ as a linear combination of iterates. Let $\zeta = \um{l}$ with $l = 1, 2, 3$, then we have
    $$
        f_c^1 (\um{l}) = (\um{l})^{2k + 1} = (\ums{l})^k \um{l} .
    $$
Hence, $\um{l} = f_c^1 (\um{l})/ a_l$, where $a = (\ums{l})^k$. To express the other element $\um{1} \um{2} \um{3}$ as a linear combination of iterates, we will adapt the trick used earlier that comes from Brouillette and Rochon's paper (see the proof of \cite[Lemma 3]{Brouillette_2019}). Consider $\zeta=t\um{1}+\um{2}+\um{3}$ with $t \in \mR \backslash \{ 0 \}$. Then we have
$$
    f_c^1 (\zeta ) = \sum_{s = 0}^{2k + 1} \sum_{r = 0}^{2k + 1 - s} \binom{2k + 1}{s} \binom{2k + 1 - s}{r} t^s (\um{1})^s (\um{2})^{r} (\um{3})^{2k + 1 - s - r} 
$$
Splitting the summation over the different parities of the integers $s$ and $r$, we obtain
    \begin{align*}
        f_c^1 (\zeta ) &= \Big( \sum_{s = 0}^k \sum_{r =0}^{r - s} \binom{2k + 1}{2s} \binom{2k - 2s + 1}{2r} t^{2s} (\ums{1})^s (\ums{2})^r (\ums{3})^{k - s - r} \Big) \um{3} \\ 
        & \phantom{=} + \Big( \sum_{s = 0}^k \sum_{r =0}^{r - s} \binom{2k + 1}{2s} \binom{2k - 2s + 1}{2r + 1} t^{2s} (\ums{1})^s (\ums{2})^r (\ums{3})^{k - s - r} \Big) \um{2} \\ 
        & \phantom{=} + \Big( \sum_{s = 0}^k \sum_{r =0}^{r - s} \binom{2k + 1}{2s+1} \binom{2k - 2s + 1}{2r} t^{2s + 1} (\ums{1})^s (\ums{2})^r (\ums{3})^{k - s - r} \Big) \um{1} \\ 
        & \phantom{=} + \Big( \sum_{s = 0}^k \sum_{r =0}^{r - s} \binom{2k + 1}{2s + 1} \binom{2k - 2s + 1}{2r + 1} t^{2s + 1} (\ums{1})^s (\ums{2})^r (\ums{3})^{k - s - r + 1} \Big) \um{1}\um{2}\um{3} 
    \end{align*}
which can be rewritten as
    $$
        f_c^1 (\zeta ) = A (t) \um{1} + B(t) \um{2} + C(t) \um{3} + D(t) \um{1} \um{2} \um{3} 
    $$
where $A(t)$, $B(t)$, $C(t)$, and $D(t)$ are polynomials in $t$. Since polynomials have only finitely many roots, we can choose a value $t_0 \in \mR$ such that $D (t_0 ) \neq 0$. Therefore, we obtain
    $$  
        \um{1}\um{2}\um{3} = \frac{1}{D (t_0)} f^1_c(\zeta) - \frac{A (t_0)}{D (t_0) a_1} f_c^1 (\um{1}) - \frac{B(t_0)}{D (t_0) a_2} f_c^1 (\um{2}) - \frac{C(t_0)}{D (t_0) a_3} f_c^1 (\um{3}) .
    $$
This shows that every basis element can be rewritten as a linear combination of iterates and therefore $\mM (\um{1}, \um{2}, \um{3} ) \subseteq \mathrm{It}^p (\um{1} , \um{2}, \um{3} )$.

Finally, we show part (3). Assume that $p$ is odd, but $c \neq 0$. The proof of part (a) is similar to the proof of parts (1) and (2). We will therefore focus on the proof of part (b). We want to show that $\mathrm{It}^p (\um{1}, \um{2}, \um{3}) = \mS (\um{1}, \um{2}, \um{3})$. 

We start by showing that $\mathrm{It}^p (\um{1}, \um{2}, \um{3}) \subseteq \mS (\um{1}, \um{2}, \um{3})$. Let $\zeta \in \mT (\um{1}, \um{2}, \um{3})$. We will do an induction on $m$ in $f_c^m (\zeta )$. Let $m = 1$ so that $f_c^1 (\zeta ) = \zeta^{2k + 1} + c$. From the proof of the base case in part (2), we know that $\zeta^{2k + 1} \in \mM (\um{1}, \um{2}, \um{3})$. This implies that $f_c^{1} (\zeta ) = \zeta^{2k + 1} + c \in \mS (\um{1}, \um{2}, \um{3})$. Now assume that the claim is true for some integer $m > 1$, which means that $f_c^m (\zeta ) \in \mS (\um{1}, \um{2}, \um{3})$. From the definition of the iterates, we have $f_c^{m + 1} (\zeta ) = (f_c^m (\zeta ))^{2k + 1} + c$. From the Induction Hypothesis, we have $f_c^m (\zeta ) \in \mS (\um{1}, \um{2}, \um{3})$. From Theorem \ref{Thm:ClosureOfSubAlgebra} part (\ref{it:ClosureofS}), the subspace $\mS (\um{1}, \um{2}, \um{3})$ is closed under multiplication and we then get $(f_c^m (\zeta ))^{2k + 1} \in \mS (\um{1}, \um{2}, \um{3})$. Hence, $f_c^{m + 1} (\zeta ) \in \mS (\um{1}, \um{2}, \um{3})$ and we conclude that $\mathrm{It}^p (\um{1}, \um{2}, \um{3}) \subseteq \mS (\um{1}, \um{2}, \um{3})$.

We now show the reverse inclusion by showing that each basis element of $\mS (\um{1}, \um{2}, \um{3})$ can be written down as a linear combination of some iterates. Notice that $f_c^1(0) = c$ and since $c \neq 0$, we have $1 = f_c^1 (0) / c$. Let $\zeta = \um{l}$ for $l = 1, 2, 3$ so that
    $$  
        f_c^1 (\um{l}) = (\ums{l})^k \um{l} + c 
    $$
and hence $\um{l} = \frac{f_c^1 (\um{l})}{a_l} - \frac{f_c^1 (0)}{a_l}$ with $a_l = (\ums{l})^k$. Let $\zeta = t_0 \um{1} + \um{2} + \um{3}$ so that, from the calculations performed at the end of the proof of part (2), we get
    $$
         \um{1}\um{2}\um{3} = \frac{1}{D (t_0)} f^1_c(\zeta) - f_c^1 (0) - \frac{A (t_0)}{D (t_0) a_1} f_c^1 (\um{1}) - \frac{B(t_0)}{D (t_0) a_2} f_c^1 (\um{2}) - \frac{C(t_0)}{D (t_0) a_3} f_c^1 (\um{3}) .
    $$
with $D_k \neq 0$. To find the linear combination to express the other basis elements, namely, $\um{1}\um{2}$, $\um{1}\um{3}$, and $\um{2}\um{3}$, we adapt for a second time the trick from Brouillette and Rochon's paper (see the proof of \cite[Lemma 3]{Brouillette_2019}). Let $t \in \mR$ with $t \neq 0$ and define $\zeta = \um{1} + t \um{2}$. Then, using the binomial theorem, we get
    $$  
        f_c^1 (\zeta ) = c + \sum_{l = 0}^{2k+1} \binom{2k+1}{l} (\um{1})^l (t\um{2})^{p - l} = c + A(t) \um{1} + B(t) \um{2}
    $$
where $A (t) = \sum_{l = 0}^{k} \binom{2k + 1}{2l + 1}(\ums{1})^l (\ums{2})^{k-l} (t^2)^{k-l}$ and $B(t) = t\sum_{l = 0}^{k} \binom{2k + 1}{2l}(\ums{1})^l (\ums{2})^{k-l} (t^2)^{k-l}$ are polynomials in $t$. Using this expression for $f_c^1 (\zeta )$, we can now write
    \begin{align*}
        f_c^2 (\zeta ) &= \Big( c + A(t) \um{1} + B(t) \um{2} \Big)^{2k + 1} + c  \\ 
        &= \sum_{l = 0}^{2k + 1} \binom{2k + 1}{l} c^l \big( A(t) \um{1} + B(t) \um{2} \big)^{2k + 1 - l} + c \\ 
        &= \sum_{l = 0}^{2k + 1} \sum_{h = 0}^{2k + 1 - l} \binom{2k + 1}{l} \binom{2k + 1 - l}{h} c^l (A(t) \um{1})^{h} (B(t) \um{2})^{2k + 1 - l - h} + c.
    \end{align*}
Splitting the sums over the parity of $l$ and $h$, we get
    \begin{align*}
        f_c^2 (\zeta ) = \Big( \sum_{l = 0}^k \sum_{h = 0}^{k - l} \binom{2k + 1}{2l} \binom{2 (k - l) + 1}{2h} c^{2l} A(t)^{2h} B(t)^{2 (k - l - h) + 1} (\ums{1})^{h} (\ums{2})^{k - l - h} \Big) \um{2} \\ 
         \phantom{=} \, + \Big(  \sum_{l = 0}^k \sum_{h = 0}^{k - l}  \binom{2k + 1}{2l} \binom{2 (k - l) + 1}{2h + 1} c^{2l} A(t)^{2h + 1} B(t)^{2 (k - l - h)} (\ums{1})^{h} (\ums{2})^{k - l - h} \Big) \um{1} \\ 
         \allowdisplaybreaks
         \phantom{=} \, + \Big(  \sum_{l = 0}^k \sum_{h = 0}^{k - l}  \binom{2k + 1}{2l + 1} \binom{2 (k - l)}{2h} c^{2l + 1} A(t)^{2h} B(t)^{2 (k - l - h)} (\ums{1})^{h} (\ums{2})^{k - l - h} \Big) + c \\ 
         \phantom{=} \, + \Big(  \sum_{l = 0}^k \sum_{h = 0}^{k - l -1}  \binom{2k + 1}{2l + 1} \binom{2 (k - l)}{2h + 1} c^{2l + 1} A(t)^{2h + 1} B(t)^{2 (k - l - h) - 1} (\ums{1})^{h} (\ums{2})^{k - l - h - 1} \Big) \um{1} \um{2}
    \end{align*}
which can be rewritten as
    $$  
        f_c^2 (\zeta ) = \tilde{A}(t) + \tilde{B} (t) \um{1} + \tilde{C}(t) \um{2} + \tilde{D}(t) \um{1}\um{2}
    $$
where $\tilde{A} (t)$, $\tilde{B} (t)$, $\tilde{C} (t)$, and $\tilde{D} (t)$ are polynomials in the variable $t$. Since any polynomial has finitely many roots, we can choose a value of $t_0 \in \mR$ ($t_0 \neq 0$) that avoids the roots of $\tilde{D} (t)$ and therefore $\tilde{D} (t_0) \neq 0$. Letting $t = t_0$, we get
    $$  
        f_c^2 (\zeta ) = \tilde{A} (t_0) \frac{f_c^1 (0)}{c} + \tilde{B} (t_0) \Big( \frac{f_c^1 (\um{1})}{a_1} - \frac{f_c^1 (0)}{a_1} \Big) + \tilde{C} (t_0) \Big( \frac{f_c^1 (\um{2})}{a_2} - \frac{f_c^1 (0)}{a_2} \Big) + \tilde{D} (t) \um{1}\um{2}
    $$
and after isolating $\um{1}\um{2}$, we get
    $$  
        \um{1}\um{2} = \frac{f_c^2 (\zeta )}{\tilde{D} (t_0)} + \Big( \frac{\frac{\tilde{B} (t_0)}{a_1} + \frac{\tilde{C} (t_0)}{a_2} - \frac{\tilde{A} (t_0) }{c}}{\tilde{D} (t_0)} \Big) f_c^1 (0) - \frac{\tilde{B} (t_0)}{a_1 \tilde{D} (t_0)} f_c^1 (\um{1}) - \frac{\tilde{C} (t_0)}{a_2 \tilde{D} (t_0)} f_c^1 (\um{2}) .
    $$
Similarly, we can write $\um{1}\um{3}$ and $\um{2} \um{3}$ as a linear combinations of iterates. This shows the inclusion $\mS (\um{1} , \um{2}, \um{3}) \subseteq \mathrm{It}^p (\um{1}, \um{2}, \um{3})$ and completes the proof of the theorem. \qed
\end{proof}

\section{Characterization of 3D slices}\label{Sec:Characterization3DSlices}
The statement of Theorem \ref{Thm:CharacterizationIterates} indicates that a space of iterates behave differently depending on the parity of the integer $p$ and on the nature of the principal units used. We will therefore start by proving our main result on the classification of principal 3D slices when $p$ is an even integer and then prove the result when $p$ is an odd integer.

\subsection{Main result for even powers}
Before proving our first main result, we make several assumptions on the principal units. It is rather easy to see that $\cT^p (\um{1}, \um{2}, \um{3}) \sim \cT^p (\sigma (\um{1}) , \sigma(\um{2}) , \sigma (\um{3}))$ where $\sigma$ is a permutation of the symbols $\um{1}$, $\um{2}$, $\um{3}$. The order of the units is therefore not important and we can choose a specific ordering. 

We choose the following ordering for a triplet of principal units $\um{1}$, $\um{2}$, $\um{3}$: $1$ has priority for the first position over imaginary units and hyperbolic units; imaginary units have priority for the second position over hyperbolic units. When all of the principal units are of the same nature, we then put them in increasing index. 

For instance, the 3D slices $\cT^p (1 , \im{1}, \jm{3})$, $\cT^p (1 ,  \jm{3}, \im{1})$, $\cT^p ( \im{1}, 1 , \jm{3})$, and $\cT^p ( \im{1}, \jm{3}, 1)$, $\cT^p (\jm{3}, 1 , \im{1})$, $\cT^p (\jm{3}, \im{1}, 1 )$ are all equivalent to $\cT^p (1 , \im{1}, \jm{3})$. Another example is the group of 3D slices $\cT^p (\im{1}, \im{2}, \im{3})$, $\cT^p (\im{1}, \im{3}, \im{2})$, $\cT^p (\im{2}, \im{1}, \im{3})$, $\cT^p (\im{2}, \im{3}, \im{1})$, $\cT^p (\im{3}, \im{1}, \im{2})$, and $\cT^p (\im{3}, \im{2}, \im{1})$. They are all equivalent to $\cT^p (\im{1}, \im{2}, \im{3})$. Without loss of generality, we will therefore assume that the principal units follow this choice of ordering.

Our first result concerning the characterization of the 3D slices of the filled-in Julia sets goes as followed. This is part 1 of Theorem \ref{Thm:main}

\begin{theorem}\label{Thm:Characterization3DSlicesEven}
    Let $n \geq 3$ be an integer and $c \in \mR$. If $p$ is an even positive integer, then there are $4$ principal 3D slices.
\end{theorem}
\begin{proof}
    Let $\um{1}, \um{2}, \um{3} \in \mI (n)$ and choose $\vm{1}, \vm{2}, \vm{3} \in \mI (n)$ such that $\ums{1} = \vms{1}$, $\ums{2} = \vms{2}$, and $\ums{3} = \vms{3}$. Such units always exist if $n \geq 3$. We will show that $\cT^p (\um{1}, \um{2}, \um{3}) \sim \cT^p (\vm{1}, \vm{2}, \vm{3})$.
    
    When $p$ is even, from Theorem \ref{Thm:CharacterizationIterates} part \eqref{Thm:CharacterizationIteratesEvenCase}, we have $\mathrm{It}^p (\um{1}, \um{2}, \um{3}) = \mL (\um{1}, \um{2}, \um{3})$ and $\mathrm{It}^p (\vm{1}, \vm{2}, \vm{3}) = \mL (\vm{1}, \vm{2}, \vm{3})$. So define the linear map $\varphi : \mathrm{It}^p (\um{1}, \um{2}, \um{3}) \ra \mathrm{It}^p (\vm{1}, \vm{2}, \vm{3})$ on the basis elements and extend it linearly: $\varphi (1) = 1$,  $\varphi (\um{1}\um{2}) = \vm{1} \vm{2}$, $\varphi (\um{1} \um{3}) = \vm{1} \vm{3}$, and $\varphi (\um{2} \um{3}) = \vm{2} \vm{3}$. Clearly $\varphi$ is a bijective linear map. 
    
     A way to show that $f_c (\varphi (\zeta )) = \varphi (f_c (\zeta))$ is to show that $\varphi (\zeta^p) = (\varphi (\zeta ))^p$. In fact, we can show that $\varphi (\zeta \eta) = \varphi (\zeta ) \varphi (\eta )$ for any $\zeta , \eta \in \mL (\um{1}, \um{2}, \um{3})$.  Let $\zeta = x + y \um{1} \um{2} + z \um{1} \um{3} + w \um{2} \um{3}$ and $\eta = X + Y \um{1}\um{2} + Z \um{1} \um{3} + W \um{2} \um{3}$. Then, we have
        \begin{align*}
            \zeta \eta &= xX + yY \ums{1} \ums{2} + zZ \ums{1} \ums{3} + wW \ums{2} \ums{3} + (xY + yX + zW \ums{3} + wZ \ums{3}) \um{1} \um{2} \\ 
            & \phantom{=} \, + (xZ + zX  + yW \ums{2} + wY \ums{2}) \um{1} \um{3} + (xW + wX + yZ \ums{1} + zY \ums{1}) \um{2} \um{3} .
        \end{align*}
    By linearity of $\varphi$ and the assumptions on the squares of each principal units, we get
         \begin{align*}
            \varphi (\zeta \eta) &= xX + yY \vms{1} \vms{2} + zZ \vms{1} \vms{3} + wW \vms{2} \vms{3} + (xY + yX + zW \vms{3} + wZ \vms{3}) \vm{1} \vm{2} \\ 
            & \phantom{=} \, + (xZ + zX  + yW \vms{2} + wY \vms{2}) \vm{1} \vm{3} + (xW + wX + yZ \vms{1} + zY \vms{1}) \vm{2} \vm{3} \\ 
            &= \varphi (\zeta ) \varphi (\eta ) .
        \end{align*}
    An induction argument on $p = 2k$ then shows that $\varphi (\zeta^p) = (\varphi (\zeta ))^p$. Using this property, we get
        \begin{align*}  
            \varphi (f_c (\zeta )) = \varphi (\zeta^p + c) = \varphi (\zeta^p) + \varphi (c) = (\varphi (\zeta ))^p + c = f_c (\varphi (\zeta )) .
        \end{align*}
    Hence, condition (2) in Definition \ref{Def:EquivalentRelation3Dslices} is satisfied. 
    
    It remains to show that condition (1) in Definition \ref{Def:EquivalentRelation3Dslices} is satisfied, since $\mT (\um{1}, \um{2}, \um{3}) \not\subset \mathrm{It}^p (\um{1}, \um{2}, \um{3})$ and $\mT (\vm{1}, \vm{2}, \vm{3}) \not\subset \mathrm{It}^p (\vm{1}, \vm{2}, \vm{3})$. For $\eta = x \um{1} + y \um{2} + z \um{3}$ for $x, y, z \in \mR$ and $\zeta = X \um{1} + Y \um{2} + Z \um{3}$ for $X, Y, Z \in \mR$, we calculate
        \begin{align*}
            \eta \zeta &= xX \ums{1} + yY \ums{2} + zZ \ums{3} + (xY + yX) \um{1} \um{2} + (xZ + zX) \um{1} \um{3} + (yZ + zY) \um{2} \um{3} .
        \end{align*}
    Applying $\varphi$ to the last equality and using linearity, we obtain 
        $$
             \varphi (\eta \zeta ) = xX \ums{1} + yY \ums{2} + zZ \ums{3} + (xY + yX) \vm{1} \vm{2} + (xZ + zX) \vm{1} \vm{3} + (yZ + zY) \vm{2} \vm{3}
        $$
    and since $\ums{1} = \vms{1}$, $\ums{2} = \vms{2}$ and $\ums{3} = \vms{3}$, we conclude that
        $$
            \varphi (\eta \zeta ) = (x \vm{1} + y \vm{2} + z \vm{3}) (X \vm{1} + Y \vm{2} + Z \vm{3}) .
        $$
    An induction argument on $p = 2k$ then implies that $\varphi (\zeta^p) = (x \vm{1} + y \vm{2} + z \vm{3})^p$ which means that 
        $$ 
            \varphi (f_c (x \um{1} + y \um{2} + z \um{3} )) = f_c (x \vm{1} + y \vm{2} + z \vm{3})
        $$ 
    for arbitrary $x, y, z \in \mR$. Hence, condition (2) of Definition \ref{Def:EquivalentRelation3Dslices} is satisfied and the two 3D slices are equivalent.

    We therefore only need to study the sign of the squares of the principal units. There are $2$ choices for each principal units, that is $+1$ and $-1$. Since the order of the selection of the principal units $\um{1}, \um{2}, \um{3}$ is not important, this gives us $4$ choices:
        \begin{enumerate}
            \item $\ums{1} = \ums{2} = \ums{3} = 1$. In this case, all principal 3D slices are equivalent to $\cT^p (1, \jm{1}, \jm{2})$.
            \item $\ums{1} = \ums{3} = 1$ and $\ums{2} = -1$. In this case, all principal 3D slices are equivalent to $\cT^p (1 , \im{1}, \jm{1})$.
            \item $\ums{1} = 1$ and $\ums{2} = \ums{3} = -1$. In this case, all principal 3D slices are equivalent to $\cT^p (1, \im{1}, \im{2})$.
            \item $\ums{1} = \ums{2} = \ums{3} = -1$. In this case, all principal 3D slices are equivalent to $\cT^p (\im{1}, \im{2}, \im{3})$. 
        \end{enumerate}
    This ends the proof. \qed
\end{proof}
We notice that all principal 3D slices are obtained from units in the set of tricomplex numbers. We therefore obtain the following consequence of our result.

\begin{corollary}
    Let $n \geq 3$ et $p \geq 2$ be integers and $c \in \mR$. For any principal 3D slice $\cT^p$ of $\cK_{n, c}^p$, there exists a principal 3D slice $\widetilde{\cT}^p$ of $\cK_{3, c}^p$ such that $\widetilde{\cT}^p \sim \cT^p$.
\end{corollary}

\subsection{Main result for odd powers}
We now treat the case when $p$ is an odd integer and $c \neq 0$. The case $c = 0$ is treated very similarly as in the previous section. This comes from the fact that when $c = 0$, the space of iterates is equal to $\mM (\um{1}, \um{2}, \um{3})$ and this characterization of the space of iterates is valid for any choice of principal units. We therefore end up with the same characterization as in the even case and this proves part 2 of \ref{Thm:main}.

\begin{theorem}
    Let $p$ be an odd integer and $c = 0$. Then there are $4$ distinct principal 3D slices.
\end{theorem}

It remains to prove part 3 of Theorem \ref{Thm:main}, that is when $c \neq 0$.

\begin{theorem}
    Let $p$ be an odd integer and $c \neq 0$. Then there are $8$ distinct principal slices when $n = 3$ and $9$ distinct principal 3D slices when $n \geq 4$.
\end{theorem}
\begin{proof}
    Assume that $p$ is an odd integer and $c \neq 0$. Let $\um{1}, \um{2}, \um{3} \in \mI (n)$ and choose $\vm{1}, \vm{2}, \vm{3} \in \mI (n)$ such that $\ums{1} = \vms{1}$, $\ums{2} = \vms{2}$, and $\ums{3} = \vms{3}$. Such units always exist if $n \geq 3$. We will show that $\cT^p (\um{1}, \um{2}, \um{3}) \sim \cT^p (\vm{1}, \vm{2}, \vm{3})$, but under additional assumptions on the principal units. We therefore split the proof into $3$ cases.
    \begin{enumerate}
        \item Assume that $\um{1} = 1$. Theorem \ref{Thm:CharacterizationIterates} part (\ref{Thm:CharacterizationIteratesOddCNotZeroClosed}) implies that $\mathrm{It}^p (\um{1}, \um{2}, \um{3}) = \mM (1, \um{2}, \um{3})$. Therefore, assume further that $\vm{1} = 1$. Define $\varphi : \mM (1, \um{2}, \um{3}) \ra \mM (1, \vm{2}, \vm{3})$ on the basis elements and by linearity: $\varphi (1) = 1$, $\varphi (\um{2}) = \vm{2}$, $\varphi (\um{3}) = \vm{3}$, and $\varphi (\um{2}\um{3}) = \vm{2} \vm{3}$. In a very similar way as in the proof of Theorem \ref{Thm:Characterization3DSlicesEven}, we can show that $\varphi (\zeta \eta ) = \varphi (\zeta ) \varphi (\eta )$ for any $\zeta , \eta \in \mM (1, \um{2}, \um{3})$ so that $\varphi (\zeta^p) = (\varphi (\zeta ))^p$ for any $\zeta \in \mM (1, \um{2}, \um{3})$. Hence, we get $\varphi (f_c (\zeta )) = f_c (\varphi (\zeta ))$. This implies that $\cT^p (\um{1}, \um{2}, \um{3}) \sim \cT^p (\vm{1}, \vm{2}, \vm{3})$ because condition (1) in Definition \ref{Def:EquivalentRelation3Dslices} is automatically verified from the fact that condition (2) is satisfied and $\mT (1, \um{2} , \um{3}) \subset \mM (1, 
        \um{2} , \um{3})$ and $\mT (1 , \vm{2} , \vm{3}) \subset \mM (1 , \vm{2} , \vm{3} )$. There are therefore three cases:
            \begin{enumerate}
                \item $\ums{2} = \ums{3} = 1$. In this case, all 3D slices are equivalent to $\cT^p (1 , \jm{1}, \jm{2})$.
                \item $\ums{2} = -1$ and $\ums{3} = 1$. In this case, all 3D slices are equivalent to $\cT^p (1, \im{1}, \jm{1})$.
                \item $\ums{2} = \ums{3} = -1$. In this case, all 3D slices are equivalent to $\cT^p (1, \im{1}, \im{2})$. 
            \end{enumerate}
        \item Assume that $\um{1} \neq 1$, but $\um{1} \um{2} = \pm \um{3}$. Theorem \ref{Thm:CharacterizationIterates} part (\ref{Thm:CharacterizationIteratesOddCNotZeroClosed}) implies that $\mathrm{It}^p (\um{1}, \um{2}, \um{3}) = \mM (1, \um{1}, \um{2})$. Therefore, assume further that $\vm{1} \vm{2} = \pm \vm{3}$. Define $\varphi : \mM (1, \um{1}, \um{2} ) \ra \mM (1, \vm{1}, \vm{3})$ on the basis elements and extend it by linearity: $\varphi (1) = 1$, $\varphi (\um{1}) = \vm{1}$, $\varphi (\um{2}) = \vm{2}$, and $\varphi (\um{1}\um{2}) = \vm{1} \vm{2}$. In a very similar way as in the proof of Theorem \ref{Thm:Characterization3DSlicesEven}, we can show that $\varphi (\zeta \eta ) = \varphi (\zeta ) \varphi (\eta )$ for any $\zeta , \eta \in \mM (1, \um{1}, \um{2})$ so that $\varphi (\zeta^p) = (\varphi (\zeta ))^p$ for any $\zeta \in \mM (1, \um{1}, \um{2})$. Hence, we get $\varphi (f_c (\zeta )) = f_c (\varphi (\zeta ))$. This implies that $\cT^p (\um{1}, \um{2}, \um{3}) \sim \cT^p (\vm{1}, \vm{2}, \vm{3})$ because condition (1) in Definition \ref{Def:EquivalentRelation3Dslices} is automatically verified from the fact that condition (2) is satisfied and $\mT (1, \um{2} , \um{3}) \subset \mM (1, 
        \um{2} , \um{3})$ and $\mT (1 , \vm{2} , \vm{3}) \subset \mM (1 , \vm{2} , \vm{3} )$. Since $\um{1} \um{2} = \pm \um{3}$, the nature of the principal units $\um{1}$ and $\um{2}$ forces a specific type for $\um{3}$. This gives three cases:
            \begin{enumerate}
                \item $\ums{1} = \ums{2} = 1$. In this case, $(\pm \um{3})^2 = \ums{1} \ums{2} = 1$. Hence $\um{3}$ is an hyperbolic unit and all 3D slices are equivalent to $\cT^p (\jm{1}, \jm{2}, \jm{3})$.
                \item $\ums{1} = -1$ and $\ums{2} = 1$. In this case, $(\pm \um{3})^2 = \ums{1} \ums{2} = -1$. Hence $\um{3}$ is an imaginary unit and all 3D slices are equivalent to $\cT^p ( \im{1}, \im{2}, \jm{1})$.
                \item $\ums{1} = \ums{2} = -1$. In this case, $\um{3}$ is a hyperbolic unit. Therefore, all 3D slices are equivalent to $\cT^p (\im{1}, \im{2}, \jm{1})$. This is case (b).
            \end{enumerate}
        \item Assume that $\um{1} \neq 1$ and $\um{1} \um{2} \neq \pm \um{3}$. Theorem \ref{Thm:CharacterizationIterates} part (\ref{Thm:CharacterizationIteratesOddCNotZeroNotClosed}) implies that $\mathrm{It}^p (\um{1}, \um{2}, \um{3}) = \mS (\um{1}, \um{2}, \um{3})$. Assume further that $\vm{1} \neq 1$ and $\vm{1} \vm{2} \neq \pm \vm{3}$. 
        
        Define $\varphi : \mS (\um{1}, \um{2}, \um{3}) \ra \mS (\vm{1}, \vm{2}, \vm{3})$ by the following formula:
            \begin{align*} 
                \varphi (x_1 + x_2 \um{1} + x_3 \um{2} + x_4 \um{3} + x_5 \um{1} \um{2} + x_6 \um{1} \um{3} + x_7 \um{2} \um{3} + x_8 \um{1} \um{2} \um{3} ) \\ 
                = x_1 + x_2 \vm{1} + x_3 \vm{2} + x_4 \vm{3} + x_5 \vm{1} \vm{2} + x_6 \vm{1} \vm{3} + x_7 \um{2} \vm{3} + x_8 \vm{1} \vm{2} \vm{3} 
            \end{align*}
        In a very similar way as in the proof of Theorem \ref{Thm:Characterization3DSlicesEven}, we can show that $\varphi (\zeta \eta ) = \varphi (\zeta ) \varphi (\eta )$ for any $\zeta , \eta \in \mS (\um{1}, \um{2}, \um{3})$ so that $\varphi (\zeta^p) = (\varphi (\zeta ))^p$ for any $\zeta \in \mS (\um{1}, \um{2}, \um{3})$. Hence, we get $\varphi (f_c (\zeta )) = f_c (\varphi (\zeta ))$. This implies that $\cT^p (\um{1}, \um{2}, \um{3}) \sim \cT^p (\vm{1}, \vm{2}, \vm{3})$. This gives four cases:
            \begin{enumerate}
                \item $\ums{1} = \ums{2} = \ums{3} = 1$. This case is impossible when $n = 3$ because all hyperbolic units are such the product of two is equal to the third. For $n > 3$, choose $\vm{1} = \im{1} \im{2}$, $\vm{2} = \im{1} \im{3}$, and $\vm{3} = \im{1} \im{4}$. Then all 3D slices are equivalent to that particular 3D slice $\cT^p (\vm{1}, \vm{2}, \vm{3})$.
                \item $\ums{1} = -1$ and $\ums{2} = \ums{3} = 1$. In this case, choose $\vm{1} = \im{1}$, $\vm{2} = \jm{1}$, and $\vm{3} = \jm{2}$, then all 3D slices are equivalent to $\cT^p ( \im{1}, \im{2}, \jm{2})$.
                \item $\ums{1} = \ums{2} = -1$ and $\ums{3} = 1$. In this case, choose $\vm{1} = \im{1}$, $\vm{2} = \im{2}$, and $\vm{3} = \jm{2}$, then all 3D slices are equivalent to $\cT^p ( \im{1}, \im{2}, \jm{2})$.
                \item $\ums{1} = \ums{2} = \ums{3} = -1$. In this case, choose $\vm{1} = \im{1}$, $\vm{2} = \im{2}$, and $\vm{3} = \im{3}$, then all 3D slices are equivalent to $\cT^p ( \im{1}, \im{2}, \im{3})$.
            \end{enumerate}
        After compiling all the possible cases, we see that when $n= 3$, there are $8$ principal 3D slices and when $n \geq 4$, there are $9$ principal 3D slices. 
    \end{enumerate}
    This concludes the proof. \qed
\end{proof}

In figure \ref{Fig:3DPSlices-Specialcase}, we illustrate all the possible principle 3D slices when $n = 3$, $p = 3$ and $c = 0.25$.

\begin{figure}
    \centering
    \begin{subfigure}{0.2\textwidth}
        \includegraphics[width=\textwidth]{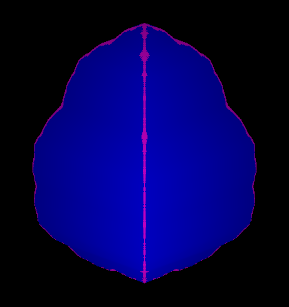}
        \caption{$\cT^3(1,\im{1},\im{2})$}
        \label{fig:octahedron1}
    \end{subfigure}
    \begin{subfigure}{0.2\textwidth}
        \includegraphics[width=\textwidth]{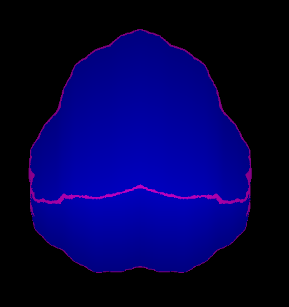}
        \caption{$\cT^3(1,\im{1},\jm{1})$}
        \label{fig:octahedron2}
    \end{subfigure}
    \begin{subfigure}{0.2\textwidth}
        \includegraphics[width=\textwidth]{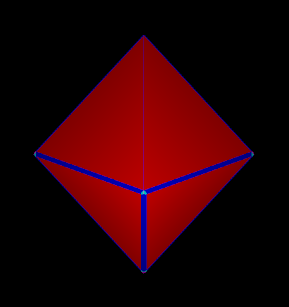}
        \caption{$\cT^3(1,\jm{1},\jm{2})$}
        \label{fig:octahedron3}
    \end{subfigure}
    \begin{subfigure}{0.2\textwidth}
        \includegraphics[width=\textwidth]{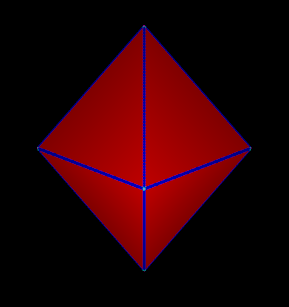}
        \caption{$\cT^3(\im{1}, \im{2}, \im{3})$}
        \label{fig:octahedron4}
    \end{subfigure}
    \vspace*{0.2cm}

    \begin{subfigure}{0.2\textwidth}
        \includegraphics[width=\textwidth]{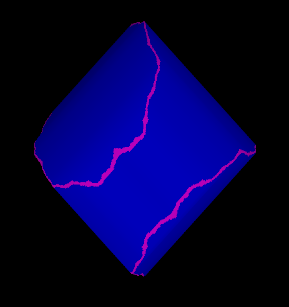}
        \caption{$\cT^3(\im{1}, \im{2}, \jm{1})$}
        \label{fig:octahedron5}
    \end{subfigure}
    \begin{subfigure}{0.2\textwidth}
        \includegraphics[width=\textwidth]{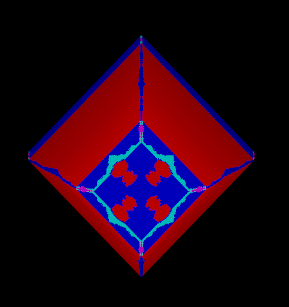}
        \caption{$\cT^3(\im{1}, \im{2}, \jm{2})$}
        \label{fig:octahedron6}
    \end{subfigure}
    \begin{subfigure}{0.2\textwidth}
        \includegraphics[width=\textwidth]{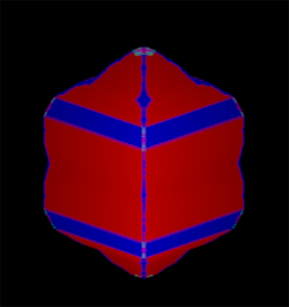}
        \caption{$\cT^3(\im{1}, \jm{1}, \jm{2})$}
        \label{fig:octahedron7}
    \end{subfigure}
    \begin{subfigure}{0.2\textwidth}
        \includegraphics[width=\textwidth]{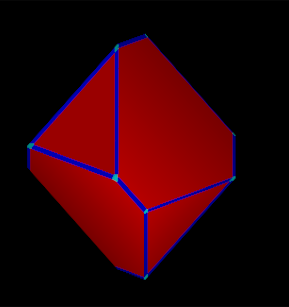}
        \caption{$\cT^3(\jm{1}, \jm{2}, \jm{3})$}
        \label{fig:octahedron8}
    \end{subfigure}
    \caption{Graphical results of principal 3D slices of the filled-in Julia set $\cK_{3, 0.25}^3$}\label{Fig:3DPSlices-Specialcase}
\end{figure}

\section{Conclusion}

In summary, we have introduced the multicomplex numbers and their topology to investigate the filled-in Julia sets in higher dimensions. The multidimensional version of the filled-in Julia sets was introduced in Section \ref{Sec:MulticomplexFilledJ} and an algorithm was presented to visualize them with a computer. In Section \ref{Sec:DefinitionEquivalence3D}, we introduced an equivalence relation as our basis for the analysis of the principal 3D slices. We then proved our main results on the number of 3D slices of the multicomplex filled-in Julia set in section \ref{Sec:Characterization3DSlices}. 

As a possible future direction of work, it would be interested to investigate in more details each equivalent classes obtained in our characterization. For instance,  based on preliminary computer explorations, the principal 3D slice $\cT^p (1, \jm{1}, \jm{2})$ seems to always be an octahedron for any values of $c \in \cM^p \cap \mR$. Is that always the case? Some further mathematical investigations are necessary to elucidate this question.


%
%

\bibliographystyle{spmpsci}
\bibliography{reference.bib}

\end{document}